\documentclass[12pt]{article}

\usepackage{color}
\usepackage{caption,subfigure}

\usepackage{amsmath,mathrsfs}
\usepackage{amssymb,amsthm}
\usepackage{nicefrac}

\usepackage{graphicx}
\usepackage{paralist}

\newtheorem{theorem}{Theorem}[section]
\newtheorem{corollary}[theorem]{Corollary}
\newtheorem{proposition}[theorem]{Proposition}

\newtheorem{lemma}[theorem]{Lemma}
\theoremstyle{definition}
\newtheorem{definition}{Definition}[section]

\newtheorem{remark}[theorem]{Remark}

\renewcommand{\qed}{\hfill $\blacksquare$}
\newcommand{\myemph}[1]{\emph{#1}}
\newcommand{\g}{\mathfrak{g}}
\newcommand{\h}{\mathfrak{h}}
\renewcommand{\a}{\mathfrak{a}}
\newcommand{\pp}{\mathfrak{p}}

\newcommand{\ad}{\operatorname{ad}}

\renewcommand{\sl}{\mathfrak{sl}}

\renewcommand{\t}{\mathfrak{t}}

\newcommand{\R}{\mathbb{R}}

\newcommand{\mf}{\mathfrak}

\newcommand{\tr}{\operatorname{tr}}
\newcommand{\C}{\mathbb{C}}

\begin{document}

\title{Distinguished Sets of Semi-simple Lie algebras}
\date{}
\maketitle

\vspace{-2cm}
\begin{center}
{\small {Xudong Chen\footnote[1]{Corresponding author. X. Chen is with the ECEE Department, University of Colorado, Boulder. {\em Email: xudong.chen@colorado.edu}.} 
\qquad \quad 
Bahman Gharesifard\footnote[2]{B. Gharesifard is with Department of Mathematics and Statistics, Queen's University, Canada. {\em Email: bahman.gharesifard@queensu.ca}.}}}
\end{center}

\begin{abstract}                
We call a finite, spanning set of a semi-simple real Lie algebra a {\em distinguished set} if it satisfies the following property: The Lie bracket of any two elements out of the set is, up to some constant, another element in the set; conversely, for any element in the set, there are two elements out of the set whose Lie bracket is, up to some constant, the given element. We show in the paper that every semi-simple real Lie algebra has a distinguished set.  
\end{abstract}


\section{Introduction and Main Result}

Let~$\mf{g}$ be a semi-simple complex Lie algebra, and $\g_0$ be a real form of $\g$. We start with the following definition:

\begin{definition}\label{def:distinguishedset}
A finite, spanning subset $S=\{X_1,\ldots,X_m\}$ of $\g_0$ is {\bf distinguished} if for any pair $(X_i,X_j) \in S\times S$, there exists an $X_k\in S$ and a constant $\lambda\in \R$ such that 
\begin{equation}\label{eq:relation2}
[X_i,X_j] = \lambda X_k,
\end{equation} 
and conversely, for any $X_k\in A$, there exists a pair $(X_i,X_j) \in S \times S $ and a {\em nonzero} constant  $\lambda \in \R$ such that~\eqref{eq:relation2} holds. 
\end{definition}

Distinguished sets find applications in control and estimation of continuum ensembles of nonholonomic systems (see, for example,~\cite{chen2019structure,li2009ensemble}).  
We establish in the paper the following result: 

\begin{theorem}\label{thm:basisliealg}
Every semi-simple real Lie algebra has a distinguished set. 
\end{theorem} 

\subsection{Distinguished sets for complex Lie algebras}
 Denote by $\ad_X(\cdot) := [X,\cdot]$ the adjoint action, and $B(X, Y) := \tr(\ad_X\ad_Y)$ the Killing form. Let $\h$ be a Cartan subalgebra of~$\g$, and $\Delta$ 
 be the set of roots. For each $\alpha\in \Delta$, we let $h_\alpha\in \h$ be such that 
$\alpha(H) =  B(h_\alpha, H)$. Denote by $\langle \alpha, \beta\rangle := B(h_\alpha, h_\beta)$, which is an inner-product defined over the $\R$-span of $\Delta$. Define $|\alpha|:= \sqrt{\langle \alpha, \alpha \rangle}$ and $H_\alpha :=  \nicefrac{2h_\alpha}{ |\alpha|^2 }$. 
For a root $\alpha\in \Delta$, let $\g_\alpha$ be the corresponding root space over $\C$.

Suppose, for the moment, that one aims to obtain a distinguished set for the complex Lie algebra $\mf{g}$; then, such a set can easily be obtained from the Chevalley basis~\cite[Ch. VII]{humphreys2012introduction}, which we recall below: 

\begin{lemma}\label{theorem:rootspacedecomp}
 There are $X_\alpha\in \g_\alpha$, for $\alpha\in \Delta$, such that the followings hold:
 \begin{enumerate}
 \item For any $\alpha\in \Delta$, we have $[X_\alpha, X_{-\alpha}] = H_\alpha$ with $B(X_\alpha, X_{-\alpha}) = \nicefrac{2}{|\alpha|^2}$.
 \item For any two non-proportional roots $\alpha, \beta$, we let $\beta + n\alpha$, with $-q \le n \le p$, be the $\alpha$-string containing $\beta$. Then,  
 $$ 
 [X_\alpha, X_\beta] = \left\{
 \begin{array}{ll}
 c_{\alpha, \beta} X_{\alpha  + \beta} & \mbox{if } \alpha + \beta \in \Delta, \\
 0 & \mbox{otherwise}.
 \end{array}
 \right.
 $$
 where  $c_{\alpha, \beta}\in \mathbb{Z}$ with $c_{\alpha,\beta} = -c_{-\alpha, -\beta}$ and $c^2_{\alpha,\beta} = (q + 1)^2$. 
 \end{enumerate}\,
\end{lemma} 


For any $\alpha, \beta\in \Delta$, $[H_\alpha, X_\beta] = \nicefrac{2\langle \alpha, \beta\rangle}{|\alpha |^2} X_\beta$ and $\nicefrac{2\langle \alpha, \beta\rangle}{|\alpha |^2}\in \mathbb{Z}$. 
It follows from Lemma~\ref{theorem:rootspacedecomp} that $S:= \{ H_\alpha, X_\alpha, X_{-\alpha} \mid \alpha\in\Delta \}$ is a distinguished set of~$\g$.  

The complex Lie algebra $\g$ can also be viewed as a Lie algebra over $\R$, which we denote by $\g^{\R}$. 
We call such a real Lie algebra {\em complex}. Since $\nicefrac{2\langle \alpha, \beta\rangle}{|\alpha|^2}$ and $c_{\alpha, \beta}$ are integers, the set $S\cup \mathrm{i} S$ is a distinguished set of $\g^{\R}$.   
Also, note that if $\g_0$ is a {\em split real form} of $\g$ (i.e.,  the $\R$-span of $S$), then $S$ is a distinguished set of $\g_0$.   
Thus, for the remainder of the paper, we will consider only the case where $\g_0$ is simple,  noncomplex, and is not a split real form of~$\g$. 

 


\subsection{Preliminary results from structure theory}


Let $\theta$ be a Cartan involution of~$\g_0$. Let $\g_0=\mathfrak{k}_0\oplus \pp_0 $ where $\mathfrak{k}_0$ (resp. $\pp_0$) is the $+1$-eigenspace (resp. $-1$-eigenspace) of $\theta$. Let $\mathfrak{k}$ (resp.~$\pp$) be the complexification of $\mathfrak{k}_0$ (resp. $\pp_0$). 
Let $\h_0$ be a $\theta$-stable Cartan subalgebra of~$\g_0$. Decompose $ \h_0=\a_0\oplus \t_0 $, where $ \a_0 $ and $ \t_0 $ are subspaces of $ \pp_0 $ and $ \mathfrak{k}_0 $, respectively.  
All roots of $\Delta$ take real value on $\a_0\oplus i\t_0 $.  
A root is said to be \emph{imaginary} (resp. \emph{real}) if it takes imaginary (resp. real) value on~$\h_0$, and hence vanishes over $\a_0$ (resp. $\t_0$). If a root is neither real nor imaginary, then it is said to be {\em complex}.
Further, $\h_0$ is said to be {\em maximally compact} if there is no real root in $\Delta$. One can obtain such $\h_0$ by Cayley transformation~\cite[Sec.~VI-7]{AWK:02}. We assume in the sequel that~$\h_0$ is maximally compact.

 
The Cartan involution permutes roots: for any root $\alpha$, let $(\theta\alpha)(H) := \alpha (\theta H)$ for any $H\in \h$. Note that $(\theta\alpha) (H) = B(H_\alpha, \theta H) = B(\theta H_\alpha,  H)$ and, hence, $H_{\theta \alpha} = \theta H_\alpha$.    
Also, note that if $\alpha$ is imaginary, then $\theta \alpha = \alpha$. Thus, $\g_\alpha$ is $\theta$-stable. Since $\g_\alpha$ is one dimensional, it is contained in either $\mathfrak{k}$ or $\pp$. An \emph{imaginary root} $ \alpha$ is said to be \myemph{compact} (resp. {\em non-compact}) if $ \g_{\alpha} \subseteq \mathfrak{k} $ (resp. $ \g_{\alpha} \subseteq\pp $).  In particular, if  $\alpha$  is compact (resp. non-compact), then $\theta X_\alpha = X_\alpha$ (resp. $\theta X_\alpha = -X_{\alpha}$). Furthermore, we can assume that  
$\theta X_\alpha = X_{\theta \alpha}$ for any complex root~$\alpha$~\cite[Thm 6.88]{AWK:02}. 
Now, for any root $\alpha$, we define $\sigma_\alpha\in \{1,-1\}$ such that $\theta X_\alpha = \sigma_\alpha X_{\theta \alpha}$. Then,  
 \begin{equation}\label{eq:cond2}
 \sigma_\alpha = \left\{
 \begin{array}{ll}
 1  & \mbox{ if } \alpha \mbox{ is  complex or compact imaginary}, \\ 
-1 & \mbox{ if } \alpha \mbox{ is  non-compact imaginary}. 
 \end{array}
 \right. 
  \end{equation} 
Note that $\sigma_\alpha = \sigma_{-\alpha} = \sigma_{\theta \alpha} = \sigma_{-\theta\alpha}$ for any root $\alpha$. Also, note that 
$$
c_{\theta\alpha, \theta\beta} X_{\theta (\alpha + \beta)} = [X_{\theta\alpha}, X_{\theta\beta}] = \sigma_\alpha\sigma_\beta \theta[X_\alpha, X_\beta] =  \sigma_\alpha\sigma_\beta c_{\alpha, \beta} \theta X_{\alpha + \beta},
$$
and, hence, 
$
c_{\theta\alpha,\theta \beta } = \sigma_\alpha\sigma_\beta \sigma_{\alpha + \beta} c_{\alpha, \beta}$. 

We provide below a few preliminary results that are necessary for constructing the distinguished sets of semi-simple real Lie algebras. We first have the following fact:

\begin{lemma}\label{prop:1}
 There is a re-scaling of the $X_\alpha\in \g_\alpha$, for $\alpha\in \Delta$, by complex numbers such that the following elements: 
$$Y_\alpha:= X_{\alpha} - \theta X_{-\alpha}, \quad Z_\alpha:= {\rm  i} (X_\alpha + \theta X_{- \alpha})$$  belong to $\g_0$ for all $\alpha \in \Delta$. 
  \end{lemma}

 \begin{proof}
Let $X_\alpha$, for $\alpha\in\Delta$, satisfy items of Lemma~\ref{theorem:rootspacedecomp}. 
We show that $\theta \overline{X_\alpha}\in \g_{-\alpha}$. 
For any $H\in \h$, we have $[H, X_{\alpha}] = \alpha(H) X_{\alpha}$. Taking complex conjugate and applying $\theta$, we obtain that  
$
[\theta \overline{H}, \theta \overline{X_{\alpha}}] = \overline{\alpha(H)} \theta \overline{X_{\alpha}}
$. Since $\alpha$ takes real value on $\a_0 + \mathrm{i}\t_0$, it follows that $\overline{\alpha(H)} = -\alpha(\theta \overline{H})$ and, hence, 
$
[\theta \overline{H}, \theta \overline{X_{\alpha}}] = -\alpha(\theta \overline{H}) \theta \overline{X_{\alpha}}
$. Because $H\mapsto \theta \overline{H}$ is an automorphism of $\h$, 
$[H, \theta \overline{X_{\alpha}}] = -\alpha(H) \theta\overline{X_{\alpha}}$ for all $H\in \h$. 

For any $X, Y\in \g_0$, we define $B_\theta(X, Y):= -B(X, \theta Y)$, which is an inner-product on $\g_0$. We further extend $B_\theta$ to $\g$ by  $B_\theta(X, Y):= -B(X, \theta \overline{Y})$ for $X, Y\in \g$. In particular, $B(X_{\alpha}, \theta \overline{X_\alpha}) < 0$. 
On the other hand,  from the first item of Lemma~\ref{theorem:rootspacedecomp}, 
we have $B(X_\alpha, X_{-\alpha}) = \nicefrac{2}{|\alpha |^2} > 0$. Since $\g_{-\alpha}$ is one-dimensional (over $\C$) and both $\theta\overline{X_{\alpha}}$ and $X_{-\alpha}$ belong to $\g_{-\alpha}$,  it follows that $\theta\overline{X_{\alpha}} = -r_\alpha X_{-\alpha}$ for some  $r_\alpha > 0$. 
Let  
$X'_{\alpha} := \sqrt{r_\alpha}^{-1} X_{\alpha}$ 
and $X'_{-\alpha} := \sqrt{r_\alpha} X_{-\alpha}$. Thus, $\theta \overline{X'_{\alpha}} = -X'_{-\alpha}$. 
We next note that 
$$
-\frac{2r_\alpha}{|\alpha |^2} = B(X_\alpha, \theta \overline{X_\alpha}) = B(X_\alpha, \overline{X_{\theta\alpha}}) = B(\theta X_\alpha, \theta \overline{X_{\theta\alpha}}) = -\frac{2r_{\theta\alpha}}{|\theta\alpha |^2}. 
$$ 
It follows that $r_\alpha = r_{\theta\alpha}$ and, hence, $\theta X'_\alpha = \sqrt{\nicefrac{r_\alpha}{r_{\theta\alpha}}}\sigma_\alpha X'_{\theta \alpha}= \sigma_\alpha X'_{\theta \alpha}$.   

Further, note that the $X'_\alpha$'s satisfy the items of Lemma~\ref{theorem:rootspacedecomp}.  
To see this, it suffices to show that $\nicefrac{r_\alpha r_\beta}{r_{\alpha + \beta}} = 1$ for $\alpha, \beta, \alpha + \beta$ roots in $\Delta$.  
We have 
\begin{multline*}
\sqrt{\frac{r_\alpha r_\beta}{r_{\alpha + \beta}}} c_{\alpha, \beta}X'_{\alpha + \beta} = [X'_{\alpha}, X'_\beta] = [-\theta\overline{X'_{-\alpha}}, -\theta\overline{X'_{-\beta}}]  =\\
\theta\overline{[X'_{-\alpha}, X'_{-\beta}]}  = - \sqrt{\frac{r_{\alpha + \beta}}{r_\alpha r_\beta}}\, \overline{c_{-\alpha, -\beta}}\, \theta \overline{X'_{-\alpha-\beta}} =- \sqrt{\frac{r_{\alpha + \beta}}{r_\alpha r_\beta}}\, \overline{c_{-\alpha, -\beta}}\, X'_{\alpha +\beta}.
\end{multline*}
Since $c_{\alpha, \beta}$ is real and $c_{\alpha,\beta} = - c_{-\alpha,-\beta}$,  it follows that $r_\alpha r_\beta /r_{\alpha + \beta} = 1$. Thus,   
the $Y_\alpha$ and $Z_\alpha$ (with $X_\alpha$ replace with $X'_\alpha$) belong to $\g_0$.  
\end{proof}

We next have the following fact which follows from~\cite[Ch.~VI]{AWK:02}:


\begin{lemma}\label{lem:introH}
Let $\g_0$ be simple and noncomplex. Suppose that $\g_0$ is not a split real form of $\g$ and that 
$\Delta$ contains a complex root; then, the underlying root system can only be $A_n$ for $n$ odd, $D_n$, or $E_6$.  
\end{lemma}

\begin{proof}
The only connected Dynkin diagrams that admit nontrivial automorphisms are $A_n$, $D_n$ and $E_6$. But, if it is $A_n$ for $n$ even, then $\g_0$ has to be
$\sl(n + 1, \R)$ (see~\cite[Ch.~VI]{AWK:02}), which is a split real form of~$\g = \sl(n + 1, \C)$.  
\end{proof}

A consequence of Lemma~\ref{lem:introH} is then the following: 

\begin{proposition}\label{lem:introH1}
If $\alpha$ is complex, then $\alpha$ and $\theta\alpha$ are strongly orthogonal. 	
\end{proposition}

\begin{proof}
By Lemma~\ref{lem:introH}, all roots in $\Delta$ share the same length. We normalize the common length to be~$\sqrt{2}$. If $\alpha$ and $\theta \alpha$ are orthogonal, then they have to be strongly orthogonal. 
Without loss of generality, we assume that $\alpha$ is positive. Let $\alpha = \sum_{\alpha_i\in \Delta^+} n_i \alpha_i$, where the $\alpha_i$'s are simple roots. The proof is carried out by induction on $n := \sum_{i} n_i$. 

For $\alpha$ a simple root, we note that there does not exist an edge connecting $\alpha$ and $\theta\alpha$ in any Vogan diagram (the root system $A_n$ for $n$ even has been ruled out). Thus, $\alpha$ and $\theta\alpha$ are orthogonal.  
For the inductive step, we write $\alpha = \beta + \gamma$ with $\beta$ and $\gamma$ positive roots. Then, $\langle\beta, \gamma \rangle = -1$. Also, note that $\beta$ and $\gamma$ cannot be imaginary at the same time. We assume, without loss of generality, that $\beta$ is complex and, hence, $\langle \beta, \theta\beta\rangle = 0$. There are two cases:  
 If $\gamma$ is imaginary, then $\langle \alpha, \theta \alpha \rangle = 2\langle \beta, \gamma  \rangle + |\gamma|^2 = 0$.
If $\gamma$ is complex, then $\langle \gamma, \theta \gamma \rangle = 0$ and, hence, 
$\langle \alpha, \theta\alpha \rangle = 2\langle \beta, \theta \gamma \rangle$. 
Since $\alpha$ is complex, $\alpha \neq \pm \theta\alpha$. If $\langle\alpha, \theta\alpha\rangle\neq 0$, then $2\langle \beta, \theta \gamma \rangle  = \langle\alpha, \theta\alpha\rangle = \pm 1$, which is a contradiction because $\langle \beta, \theta \gamma \rangle$ has to be an integer. 
\end{proof}

The next result is a corollary of Prop.~\ref{lem:introH1} and follows from computation: 

\begin{corollary}
	If $\alpha$ is complex, then $[Y_\alpha,Z_\alpha] =0$. Moreover, for any $\alpha\in \Delta$,    
\begin{equation}\label{eq:introHHHH}
\left\{
\begin{array}{lllll}
\left [Y_\alpha, Y_{-\alpha} \right ] & = & -[Z_\alpha, Z_{-\alpha}] & = &  H_\alpha - H_{\theta \alpha}, \\
\left [Y_\alpha, Z_{-\alpha} \right ] & = & -\left [Y_{-\alpha}, Z_{\alpha} \right ] & = & \mathrm{i} (H_\alpha + H_{\theta \alpha} ).
\end{array}
\right.\,
\end{equation}
\end{corollary}

\subsection{Distinguished sets for real Lie algebras}   
Define subsets $\Phi_\alpha$, for $\alpha\in \Delta$,  of $\g_0$ as follows:
\begin{equation}\label{eq:defphi}
\Phi_\alpha:= \{H_\alpha - H_{\theta \alpha},\,\, \mathrm{i}(H_\alpha + H_{\theta \alpha}), \,\, Y_\alpha, \,\, Z_\alpha\}.
\end{equation}
where $Y_\alpha$ and $Z_\alpha$ are elements of $\g_0$ defined in Lemma~\ref{prop:1}. 

For two arbitrary subsets $S$ and $S'$ of $\g_0$, we let $[S, S']:= \{[X, X'] \mid X\in S, X'\in S'\}$. Such a notation is different from the convention that $[S, S']$ is a linear span of $[X, X']$ for $X\in S$ and $X'\in S'$. 
We now define 
$$S_0 := \cup_{\alpha\in \Delta}\Phi_\alpha \quad \mbox{and} \quad S_k := [S_{k - 1}, S_{k-1}], \quad  \forall k\ge 1.$$ 
Note that $S_0$ spans $\g_0$. Since $\g_0$ is simple, $S_k$ spans $\g_0$ for all $k\ge 0$.

We will now have the following result.

\begin{theorem}\label{thm:allcases}
	Let $\g_0$ be a simple real Lie algebra. Suppose that $\g_0$ is noncomplex and is not a split real form of~$\g$; then, the following hold: 
	\begin{enumerate}
		\item If $\Delta$ contains only imaginary roots and the underlying root system is $A_{n}$, $D_n$, or $E_m$ for $m = 6,7,8$, then $S_0$ is distinguished. 
		\item If $\Delta$ contains only imaginary roots and the underlying root system is $B_n$ or $C_n$, or, if $\Delta$ contains complex roots and  the root system is $A_{2n-1}$ or $D_n$, then $S_1$ is distinguished.   
		\item If $\Delta$ contains only imaginary roots and the underlying root system is $F_4$, or, if $\Delta$ contains complex roots and  the root system is $E_6$, then $S_2$ is distinguished. 
		\item If the underlying root system is $G_2$, then the following set: 
		\begin{equation}\label{eq:defsstar}
			S^*:=\bigcup_{\alpha\in \Delta_{\rm long}} \Phi_\alpha\cup \bigcup_{\small
			\begin{array}{l}
			\alpha, \beta\in \Delta_{\rm short} \vspace{1pt}\\
			\mbox{and } \alpha \neq \pm \beta
			\end{array}}
			\{Y_{\alpha + \beta} \pm Y_{\alpha - \beta}, \,\, Z_{\alpha + \beta} \pm Z_{\alpha - \beta}\}
		\end{equation}
 is distinguished, where $\Delta_{\rm long}$ and $\Delta_{\rm short}$ comprise long and short roots, respectively.     
	\end{enumerate}
\end{theorem}

Theorem~\ref{thm:allcases} is summarized in Table~\ref{tab:table}: 

\begin{table}[h]
\centering
\begin{tabular}{c|c|c}
		Root systems & with only imaginary roots & with complex roots \\
		\hline \hline 
	           $A_{n} $ or $ D_n $ & $S_0$& $S_1$ ($n$ is odd for $A_n$)\\
		\hline
	           $E_6$ & $S_0$& $S_2$\\		
		\hline		
	           $E_7$ or $ E_8 $ & $S_0$ & --- \\				
		\hline		
	           $B_n$ or $ C_n $ & $S_1$ & --- \\				
		\hline		
	           $F_4$  & $S_2$ & ---\\				
		\hline		
	           $G_2$ & $S^*$  & ---\\				
		\hline		
	\end{tabular}
	\caption{Distinguished sets of simple real Lie algebras $\g_0$, where $\g_0$ is noncomplex and is not a split real form of $\g$.}
	\label{tab:table}
\end{table}




\section{Single Roots and Matched Pairs}\label{sec:sm}  
\subsection{Bipartitions of root sets}\label{ssec:matchpairs}

We start with the following definition:  

\begin{definition}
A pair of roots $(\alpha, \beta)$ is {\bf matched} if both $\alpha + \beta$ and $\alpha - \theta\beta$ are roots. A root  $\alpha$ is {\bf single} if there does not exist a root $\beta$ such that $(\alpha, \beta)$ is matched.  
\end{definition}

Matches pairs and single roots will be of great use in computation of distinguished sets. We have the following result:

\begin{theorem}\label{prop:matchedpairs}
The following hold for single roots:
\begin{enumerate}
\item If $\Delta$ has only imaginary roots and the underlying root system is $A_{n}$, $D_n$, or $E_m$ for $m = 6,7,8$, then all roots are single.  
\item If $\Delta$ has only imaginary roots and the underlying root system is $B_n$, $C_n$, $F_4$, or $G_2$, then a root is single if and only if it is long.
\item If $\Delta$ has complex roots, then a root is single if and only if it is imaginary. 
\end{enumerate}
\end{theorem}

For a given $\Delta$, we let $\Delta_{\rm sh}$ and $\Delta_{\rm long}$ be the collections of short and long roots, respectively, and let $\Delta_{\rm im}$ and $\Delta_{\rm comp}$ be the collections of imaginary and complex roots, respectively.    
 Let $e_1,\ldots, e_n$ be the standard basis of $\R^n$. 
 For an arbitrary set $S$ in $\R^n$, we let $-S$ be its negative and $\pm S:= S \cup - S$.  

\begin{proof}
For item~1, we note that all roots share the same length. If $(\alpha, \beta)$ is a matched pair, then $\|\alpha \pm \beta\|^2 = 2\|\alpha\|^2$ which is a contradiction. 


For items 2 and 3, we reproduce below root systems, together with the Cartan involutions (if $\Delta$ contains complex roots). If $\Delta$ has only imaginary roots, then there are four cases:    
\begin{enumerate}
\item[(a)] Root system is $B_n$. Let $\Delta = \Delta_{\rm sh}
\cup \Delta_{\rm long}$ be defined as follows:  
$$
\left\{
\begin{array}{lll}
\Delta_{\rm sh} & := & \{ \pm e_i\mid 1\le i\le n \}, \\ 
\Delta_{\rm long} & := & \{\pm e_i \pm  e_j \mid 1\le i< j\le n \}.
\end{array}
\right. $$ 
Then, $(\alpha, \beta)$ is matched if and only if $\alpha, \beta\in \Delta_{\rm sh}$ and $\alpha \neq \pm\beta$. 

\item[(b)] Root system is $C_n$. Let $\Delta = \Delta_{\rm sh}
\cup \Delta_{\rm long}$ be defined as follows:  
$$
\left\{
\begin{array}{lll}
\Delta_{\rm sh} & := & \{\pm \nicefrac{\sqrt{2}}{2} \, e_i \pm \nicefrac{\sqrt{2}}{2}\,  e_j \mid 1\le i < j \le n \}, \\ 
\Delta_{\rm long} & := & \{ \pm \sqrt{2} e_i \mid 1\le i\le n \}.
\end{array}
\right.
$$ 
Then, 
$(\alpha, \beta)$ is matched if and only if $\pm\{\alpha, \beta\} = \{\pm \nicefrac{\sqrt{2}}{2} \, e_i \pm \nicefrac{\sqrt{2}}{2}\,  e_j \}$.  
 
\item[(c)] Root system is $F_4$. Let $\Delta = \Delta_{\rm sh}
\cup \Delta_{\rm long}$ be defined as follows: 
\begin{equation}\label{eq:F4longroots}
\Delta_{\rm long} := \{\pm e_i \pm e_j \mid 1\le i < j \le 4\},
\end{equation}
and $\Delta_{\rm sh}= \cup^3_{i = 1}\Delta_{{\rm sh}_i}$ where 
\begin{equation}\label{eq:F4shortroots}
\left\{
\begin{array}{lll}
\Delta_{{\rm sh}_1} & := & \left\{\pm e_i \mid 1\le i \le 4\right\}, \vspace{2pt}\\
 \Delta_{{\rm sh}_2} & := & \{\nicefrac{1}{2}\sum^4_{i = 1} \epsilon_i e_i \mid \epsilon^2_i = 1 \mbox{ and }  \epsilon_1\epsilon_2\epsilon_3\epsilon_4 = 1 \}, \vspace{2pt}\\
 \Delta_{{\rm sh}_3} & := & \{\nicefrac{1}{2}\sum^4_{i = 1} \epsilon_i e_i \mid \epsilon^2_i = 1 \mbox{ and }  \epsilon_1\epsilon_2\epsilon_3\epsilon_4 = -1 \}.
\end{array}
\right.
\end{equation}
Then, $(\alpha, \beta)$ is matched if and only if $\alpha, \beta$ belong to the same $\Delta_{{\rm sh}_i}$ for some $i = 1,2,3$ and $\alpha \neq \pm\beta$. 

\item[(d)] Root system is $G_2$. Let $\Delta = \Delta_{\rm sh}
\cup \Delta_{\rm long}$ be defined as follows: 
$$
\left\{
\begin{array}{lll}
\Delta_{\rm sh} & := & \pm  \{e_i - e_j \mid 1\le i < j \le 3 \}, \\ 
\Delta_{\rm long} & := & \pm \{e_i +  e_j - 2e_k \mid \{i,j,k\} = \{1,2,3\} \}.
\end{array}
\right.
$$ 
Then, $(\alpha, \beta)$ is matched if and only if $\alpha,\beta\in \Delta_{\rm sh}$ and  $\alpha \neq \pm\beta$.
\end{enumerate}
If $\Delta$ has complex roots, then there are three cases:
\begin{enumerate}
\item[(e)]  Root system is $A_{2n -1}$. Let $\Delta = \Delta_{\rm im}\cup \Delta_{\rm comp}$ be defined as follows: 
$$
\left\{
\begin{array}{lll}
\Delta_{\rm im} & := & \{e_i - e_{2n + 1 - i} \mid 1\le i \le 2n \}, \\
\Delta_{\rm comp} & := & \{e_i - e_j \mid i\neq j,i + j\neq 2n+1\}.
\end{array}
\right.
$$  
The Cartan involution $\theta$ acts on $\Delta$ as follows:  
$\theta: e_i - e_j \mapsto e_{2n + 1 - j} -e_{2n + 1 - i}$. 
Then, $(\alpha, \beta)$ is matched if and only if $(\alpha, \beta) = (e_i - e_j, e_j - e_{2n + 1 - i})$ for $j \neq 2n + 1 - i$.   

\item[(f)] Root system is $D_n$. Let $\Delta = \Delta_{\rm im}\cup \Delta_{\rm comp}$ be defined as follows: 
$$
\left\{
\begin{array}{lll}
 \Delta_{\rm im} & := & \{\pm e_i \pm e_j \mid 1 \le i < j \le n - 1\}, \\ 
 \Delta_{\rm comp} & := & \{\pm e_i \pm e_n \mid i\neq n\}.
 \end{array}
 \right. 
$$
 The Cartan involution $\theta$ acts on $\Delta$ by negating the sign of $e_n$, i.e., $\theta: e_i + e_n \mapsto e_i - e_n$. 
 Then,  $(\alpha, \beta)$ is matched if and only if 
$(\alpha, \beta) = (\epsilon_i e_i + \epsilon_n e_n, \epsilon_j e_j - \epsilon_n e_n)$ for $i\neq j$ and $\epsilon_i^2 = \epsilon_j^2 = \epsilon_n^2 = 1$. 

\item[(g)] Root system is $E_6$. First, we define $a_i,b_i\in \R^3$, for $1\le i \le 3$, as follows:    
\begin{equation}\label{eq:defaibi}
a_i := e_j - e_k \quad  \mbox{and} \quad b_i := (e_j + e_k - 2e_i)/3,
\end{equation}
 where $(i,j,k)$ is a cyclic rotation of $(1,2,3)$. 
 We then let $\Delta_{\rm im}$ and $\Delta_{\rm comp}$ be sets of vectors in $\R^9$ defined as follows: 
\begin{equation}\label{eq:E6imroots}
\Delta_{\rm im} := \pm \{(a_i, 0, 0) \mid 1 \le i \le 3\} \cup  
 \pm \{ (b_i,b_j,b_j) \mid 1\le i,j \le 3\}, 
\end{equation}
and $\Delta_{\rm comp} = \cup^3_{i = 1}\Delta_{{\rm comp}_i}$ where each $\Delta_{{\rm comp}_i}$ is given by
\begin{multline}\label{eq:E6comproots}
\Delta_{{\rm comp}_i} :=  
\pm \{(0, a_i, 0), (0,0, a_i)\} \,\, \cup \\
\pm \{ (b_{l}, b_{j}, b_{k})\mid  1\le l\le 3 \mbox{ and } \{j,k\} = \{1,2,3\} \backslash \{i\} \}. 
\end{multline}
The Cartan involution $\theta$ acts on $\Delta$ as follows:   
 $\theta: (r, s, t) \mapsto (r, t, s)$. 
 Then, $(\alpha, \beta)$ is matched if and only if  $\alpha,\beta$  belong to the same $\Delta_{{\rm comp}_i}$ for some $i = 1,2,3$ with $\beta\not\in \{\pm \alpha, \pm \theta\alpha\}$ and $\langle \alpha, \beta \rangle < 0$.
 \end{enumerate}
 This completes the proof.
\end{proof}

The following result is a corollary of Theorem~\ref{prop:matchedpairs}:

\begin{corollary}\label{cor:orthogonalpair}
	If $(\alpha, \beta)$ is a matched pair, then $\alpha  + \beta$ and $\alpha - \theta \beta$ are strongly orthogonal. 
\end{corollary}

\begin{proof}
If $\Delta$ contains only imaginary roots, then, by Theorem~\ref{prop:matchedpairs}, $\alpha,\beta\in \Delta_{\rm sh}$. It follows that $\langle \alpha + \beta, \alpha - \beta \rangle = |\alpha|^2 - |\beta|^2 = 0$. Since $2\alpha$ and $2\beta$ are not roots, $\alpha + \beta$ and $\alpha - \beta$ are strongly orthogonal. 
If $\Delta$ contains complex roots, then all roots share the same length. We normalize the length to be $\sqrt{2}$. Since $\alpha + \beta$, $\alpha - \theta \beta$ are roots, $\langle\alpha, \beta  \rangle = -1$  and $\langle\alpha, \theta\beta  \rangle = 1$. 
By Theorem~\ref{prop:matchedpairs}, $\alpha, \beta\in \Delta_{\rm comp}$. Then, by Prop.~\ref{lem:introH1}, $\langle \beta, \theta \beta\rangle = 0$, so  
$\langle \alpha + \beta, \alpha - \theta\beta \rangle =  2 + \langle \alpha, \beta\rangle + \langle \alpha, -\theta\beta\rangle = 0$. 	
Since $(\alpha - \theta\beta)$ is imaginary, $\alpha -\theta \beta = \theta\alpha - \beta$ and, hence, $(\alpha + \beta) + (\alpha - \theta\beta) = \alpha + \theta\alpha$ and $(\alpha + \beta) - (\alpha - \theta\beta) = \beta + \theta\beta$. By Prop.~\ref{lem:introH1}, $\alpha + \theta\alpha$ and $\beta + \theta\beta$ are not roots. Thus, $\alpha + \beta$ and $\alpha - \theta\beta$ are strongly orthogonal. 
\end{proof}

\subsection{Induced orthogonal pairs}

Following Corollary~\ref{cor:orthogonalpair}, we introduce the following definition:

\begin{definition}
We call a pair of roots $(x,y)$ an {\bf induced orthogonal pair (\textit{\textbf{iop}}}) if there is a matched pair $(\alpha, \beta)$ such that $x = \alpha + \beta$ and $y  = \alpha - \theta\beta$. 
\end{definition}

For an {\em iop} $(x,y)$, we let $R_{x, y} := \pm \{x, y\}$. For two {\em iop}s $(x,y)$ and $(x',y')$, we write $(x, y)\sim (x',y')$ if $R_{x, y} = R_{x',y'}$. The following result can be established using  the root systems reproduced in the proof of Theorem~\ref{prop:matchedpairs}.

\begin{theorem}\label{prop:orthogonalpairs}
Let $(x, y)$ be an {\em iop}. If $\Delta$ has only imaginary roots, then there are four cases: 
	\begin{enumerate}
	\item[(a)] Root system is $B_n$ and $(x, y) \sim (e_i+e_j, e_i - e_j)$ for $i\neq j$.
	\item[(b)] Root system is $C_n$ and $(x, y) \sim (\sqrt{2} e_i, \sqrt{2} e_j)$ for $i\neq j$. 
	\item[(c)] Root system is $F_4$ and $x, y\in \Delta_{\rm long}$ with $\langle x, y\rangle = 0$.
	\item[(d)] Root system is $G_2$ and $\langle x, y\rangle = 0$.	
	\end{enumerate}
	If $\Delta$ has complex roots, then there are three cases:
	\begin{enumerate}
   \item[(e)] Root system is $A_{2n-1}$ and $(x, y) \sim (e_i - e_{2n + 1 - i}, e_j - e_{2n + 1 - j})$ for $i\neq j$ and $i + j\neq 2n + 1$.	
   \item[(f)] Root system is $D_n$ and  $(x, y) \sim (e_i + e_j, e_i - e_j)$ for $i\neq j$, $i\neq n$, and $j\neq n$.
   \item[(g)] Root system is $E_6$ and $x,y\in \Delta_{\rm im}$ with $\langle x, y\rangle = 0$.
   \end{enumerate} 
\end{theorem}


\begin{remark}
Note that if the root system is not $G_2$ and $(x,y)$ is an {\em iop}, then $x$ and $y$ are single (and, hence, imaginary).
But, if the root system is $G_2$, then $(x,y)$ is an {\em iop} if and only if $(x, y) = (e_i - e_j, e_i + e_j - 2e_k)$ for $(i,j,k)$ a cyclic rotation of $(1,2,3)$, so one of the two roots $\{x, y\}$ is short and the other is long.   	
\end{remark}


We now assume that $\Delta$ contains at least a single root. 
If a root system is named ``$\star$'' ($\star = A_n, B_n$, etc.), then we let  $\Delta^\star_{\rm sin}$ be the collection of single roots in $\Delta$. 
There is no ambiguity of using such a notation: By Theorem~\ref{prop:matchedpairs}, if the underlying root system is $B_{n}$, $C_n$, $F_4$, or $G_2$, then $\Delta^\star_{\rm sin} = \Delta_{\rm long}$. If the underlying root system is $A_{2n - 1}$, $D_n$, or $E_6$, then $\Delta^\star_{\rm sin} = \Delta_{\rm im}$.     
Let $V^\star_{\rm sin}$ be the vector space spanned by $\Delta^\star_{\rm sin}$ equipped with the standard inner-product. We have the following fact: 

\begin{proposition}\label{prop:isometry}
Let $(\star, \star')$ be either $(B_n, D_{n + 1})$, $(C_n, A_{2n - 1})$, or $(F_4, E_6)$. Then, in each case, there is a linear isometry $\pi_{\star \to \star'}: V^\star_{\rm sin} \to V^{\star'}_{\rm sin}$ 
such that it preserves single roots, i.e., $\pi_{\star \to \star'}(\Delta^\star_{\rm sin}) = \Delta^{\star'}_{\rm sin}$. Moreover, the map preserves iops, i.e., if $(x, y)$ is an {\em iop}, then so is $(\pi_{\star \to \star'}(x), \pi_{\star \to \star'}(y))$. 
\end{proposition}

\begin{proof}
We construct the isometries explicitly. Each $\pi_{\star \to \star'}$ is defined by specifying its value on a basis of $V^\star_{\rm sin}$. For the first two maps, we let  
$$
\begin{array}{lllll}
	\pi_{B_n \to D_{n + 1}}( e_i \pm e_n ) & := & e_i \pm  e_n, & \quad & \forall \,\, 1\le i <  n, \vspace{3pt}\\
	\pi_{C_n\to A_{2n - 1}}(\sqrt{2} e_i) & := & e_i - e_{2n + 1 - i}, & \quad & \forall \,\, 1\le i \le n. \vspace{3pt} \\
\end{array}
$$
For the last map, we recall that $a_i, b_i\in \R^3$ are defined in~\eqref{eq:defaibi} and let  
$$
\left\{
\begin{array}{ll}
	\pi_{F_4\to E_6}(e_1 + e_2) := (a_1, 0, 0), & \pi_{F_4\to E_6}(e_1 - e_2) := (b_1, b_1, b_1), \vspace{3pt} \\
	\pi_{F_4\to E_6}(e_3 + e_4) := (b_1, b_2, b_2), & \pi_{F_4\to E_6}(e_3 - e_4) := (b_1, b_3, b_3).
\end{array}
\right.
$$
Computation shows that the above maps satisfy the required properties. 	
\end{proof}

The above result implies that if $(\star, \star')$ is one of the three pairs, then $\Delta^\star_{\rm sin}$ and $\Delta^{\star'}_{\rm sin}$ are essentially the same. We now take a closer look at long roots  in $F_4$ (and imaginary roots in $E_6$). 
We decompose $\Delta_{\rm sin} = \cup^3_{i = 1} \Delta_{{\rm sin}_i}$  as follows:  
\begin{equation}\label{eq:singleroots}
\left\{
\begin{array}{l}
\Delta_{{\rm sin}_1}:= \{\pm e_1 \pm e_2, \,\, \pm e_3 \pm e_4\},\\ 
\Delta_{{\rm sin}_2}:= \{\pm e_1 \pm e_4, \,\, \pm e_2 \pm e_3 \}, \\
\Delta_{{\rm sin}_3}:= \{\pm e_1 \pm e_3, \,\, \pm e_2 \pm e_4 \}.
\end{array}
\right.
\end{equation}
The decomposition is made such that two roots belong to the same $\Delta_{{\rm sin}_i}$ if and only if they are proportional or orthogonal. By item~{ (c)} of Theorem~\ref{prop:orthogonalpairs}, if $(x, y)$ is an {\em iop}, then $x, y$ belong to the same $\Delta_{{\rm sin}_i}$ for some~$i$. The same decomposition can be applied to imaginary roots in $E_6$ under the map $\pi_{F_4 \to E_6}$. With a slight abuse of notation, we will use $\Delta_{{\rm sin}_i}$ to denote either $\Delta_{{\rm sin}_i}$ for $F_4$ or $\pi_{F_4 \to E_6}(\Delta_{{\rm sin}_i})$ for $E_6$.

\subsection{Addition of induced orthogonal pairs}
We assume in the section that the underlying root system is {\em not} $G_2$. 
For two {\em iop}s 
$(x, y)$ and $(x',y')$, let  
$
Q:= R_{x,y} + R_{x',y'}
$. 
We investigate the case  where $Q$ contains either zero or at least a root. 
We first have the following result:


\begin{proposition}\label{prop:zeroQ}
	If $Q$ contains zero, then $R_{x, y}$ intersects $R_{x',y'}$. Moreover,
	 any two nonproportional roots out of $R_{x, y} \cup R_{x',y'}$ are strongly orthogonal.
	\end{proposition}

\begin{proof}
The result is trivial if $R_{x, y} = R_{x',y'}$. We thus assume that  
$R_{x, y} \cap R_{x',y'} = \{\pm z\}$
for some root~$z$. By Theorem~\ref{prop:orthogonalpairs} and the linear isometries (Prop.~\ref{prop:isometry}), it suffices to consider root systems $C_n$ and $F_4$. If root system is $C_n$, then $R_{x, y}\cup R_{x',y'} = \pm\{\sqrt{2} e_i, \sqrt{2} e_j, \sqrt{2} e_k\}$ for $i,j,k$ pairwise distinct. 
	If root system is $F_4$, then by~\eqref{eq:singleroots}, $R_{x,y}$ and $R_{x',y'}$  belong to the same $\Delta_{{\rm sin}_i}$ for some $i = 1,2,3$. 
	 and, hence, any two nonproportional roots out of $R_{x, y} \cup R_{x',y'}$ are orthogonal to each other. Since all roots in $R_{x, y} \cup R_{x',y'}$ are single, orthogonality implies strong orthogonality.  	
\end{proof}

We next consider the case where $Q$ contains at least a root. For the given roots $x$, $y$, $x'$, and $y'$, we define a product of their inner-products as follows: 
\begin{equation}\label{eq:defxi}
\xi:= \langle x, x'\rangle \langle y, y'\rangle \langle x, y'\rangle \langle y, x'\rangle.
\end{equation}
Note that $\xi$ is invariant if ones replaces any root $x,y,x',y'$ with its negative. Thus, $\xi$ depends only on $R_{x, y}$ and $R_{x',y'}$. By Theorem~\ref{prop:orthogonalpairs}, there are only two cases (up to the isometry~$\pi_{\star\to\star'}$) in which $Q$ has at least a root: 
\begin{enumerate}
\item Root system is $B_n$ (or $D_{n + 1}$) and $R_{x, y} = \{\pm e_i \pm e_j\}$, $R_{x',y'} = \{\pm e_j \pm e_k\}$ where $i,j,k$ are pairwise distinct.  The set $Q$ has four roots $R_{x'',y''}$ where $(x'',y'')$ is an {\em iop} with $(x'',y'') \sim (e_i + e_k, e_i - e_k)$. In this case,  $\xi > 0$. 

\item Root system is $F_4$ (or $E_6$) and $R_{x, y} \subset \Delta_{{\rm sin}_i}$, $R_{x',y'}\subset \Delta_{{\rm sin}_j}$ with $i\neq j$. In this case, all roots of $Q$ belong to $\Delta_{{\rm sin}_k}$ where $\{k\}:= \{1,2,3\} \backslash \{i,j\}$. Because two roots belong to the same $\Delta_{{\rm sin}_l}$ if and only if  they are proportional or orthogonal, we have that $\xi\neq 0$. However, $\xi$ can be either positive or negative (in contrast to the previous case where $\xi$ is always positive). Examination of the root system leads to the following fact:     
   $\xi > 0$ if and only if $Q$ contains four roots $R_{x'',y''}$ for $(x'',y'')$ an {\em iop}; $\xi < 0$ if and only if $Q$ contains eight roots $\Delta_{{\rm sin}_k}$.  
\end{enumerate}
The next proposition then follows directly from the above arguments:  

\begin{proposition}\label{prop:f4e6q}
Suppose that $Q$ has at least a root; then, there are two cases: 
\begin{enumerate}
	\item If $\xi > 0$, then $Q$ contains four roots $R_{x'',y''}$ where $(x'',y'')$ is an iop.
	\item If $\xi < 0$, then the underlying root system can only be $F_4$ or $E_6$. Let $i,j\in \{1,2,3\}$ be such that $R_{x, y}\subset \Delta_{{\rm sin}_i}$ and $R_{x',y'}\subset \Delta_{{\rm sin}_j}$. Then, $i \neq j$ and, moreover, $Q$ contains eight roots $\Delta_{{\rm sin}_k}$ where $\{k\} := \{1,2,3\}\backslash \{i,j\}$.   
\end{enumerate}
\end{proposition}

\section{Computation of Distinguished Sets}\label{sec:classicalim}

We establish here Theorem~\ref{thm:allcases}. 
We assume in Sections~\ref{ssec:computes0} -- \ref{ssec:computes3} that the underlying root system is {\em not} $G_2$ and establish items~(1) -- (3) of the theorem. We deal with the case $G_2$ in Section~\ref{ssec:caseG2} and establish item~(4). 

\subsection{Computation of $S_1$}\label{ssec:computes0}
For two arbitrary subsets $S$ and $S'$ of $\g_0$, we write $S \sqsubseteq S'$ if for any $X\in S$, there is an $X'\in S'$ such that $X = \lambda X'$ for some $\lambda\in \R$. We write $S \equiv S'$ if $S\sqsubseteq S'$ and $S' \sqsubseteq S$. 

  Recall that $\Phi_\alpha = \{H_\alpha - H_{\theta \alpha}, \mathrm{i}(H_\alpha + H_{\theta\alpha}), Y_\alpha, Z_\alpha\}$ and $S_0 = \cup_{\alpha\in\Delta}\Phi_\alpha$. If $\gamma$ is not a root, then by default $X_\gamma$ (and, hence, $Y_\gamma = Z_\gamma = 0$). 
The following result directly follows from computation: 
 
\begin{proposition}\label{lem3}
For any $\alpha, \beta\in \Delta$,   
\begin{equation}\label{lem3:cond11}
\left\{
\begin{array}{lll}
 [H_\alpha - H_{\theta \alpha}, Y_{\beta} ]  =  \nicefrac{2\langle \alpha, \beta -\theta\beta \rangle}{|\alpha|^2} Y_\beta, &  & 
 { [\mathrm{i}(H_\alpha + H_{\theta\alpha}), Y_\beta  ]}  =  \nicefrac{2\langle \alpha, \beta + \theta\beta \rangle}{|\alpha|^2}  Z_\beta, \vspace{3pt} \\
 { [H_\alpha - H_{\theta \alpha}, Z_{\beta}  ]}  =  \nicefrac{2\langle \alpha, \beta -\theta\beta \rangle}{|\alpha|^2} Z_\beta, &  & 
 { [\mathrm{i}(H_\alpha + H_{\theta\alpha}), Z_\beta  ]}  =  -\nicefrac{2\langle \alpha, \beta + \theta\beta \rangle}{|\alpha|^2}  Y_\beta. 
\end{array}
\right.
\end{equation}
If $\beta \notin \{\pm\alpha, \pm \theta \alpha\}$,  then
\begin{equation}\label{lem3:cond1}
\left\{
\begin{array}{lll}
{ [Y_\alpha, Y_\beta  ]} & = & c_{\alpha, \beta} Y_{\alpha + \beta} - \sigma_\beta c_{\alpha,-\theta\beta} Y_{\alpha - \theta \beta}, \\
{ [Z_\alpha, Z_\beta  ]} & = & -c_{\alpha, \beta} Y_{\alpha + \beta} -\sigma_\beta c_{\alpha,-\theta\beta} Y_{\alpha - \theta \beta}, \\
{ [Y_\alpha, Z_{\beta} ]} & = & c_{\alpha, \beta} Z_{\alpha + \beta} + \sigma_\beta c_{\alpha, -\theta\beta} Z_{\alpha -\theta\beta}, \\
{ [ Z_{\alpha}, Y_\beta ]} & = & c_{\alpha, \beta} Z_{\alpha + \beta} - \sigma_\beta c_{\alpha, -\theta\beta} Z_{\alpha -\theta\beta}.  
\end{array}
\right.
\end{equation}
If $\beta \in \{-\alpha, \theta \alpha\}$, then  
$[\Phi_\alpha, \Phi_\beta] \equiv \Phi_\alpha$. If $\beta\in \{\alpha, -\theta\alpha\}$ and $\alpha$ is complex, then $[\Phi_\alpha, \Phi_\beta] \equiv \{Y_\alpha, Z_\alpha\}$. 
\end{proposition}

By Prop.~\ref{lem3}, $S_0 \sqsubseteq S_1$ and, hence, $S_k \sqsubseteq S_{k+1}$ for any $k\ge 0$.

 Now, let $(\alpha, \beta)$ be a matched pair. Since the underlying root system is not $G_2$, it follows that $c^2_{\alpha, \beta} = c^2_{\alpha, -\theta\beta}$. 
 Let $(x, y)$ be the corresponding {\em  iop}, i.e., $x = \alpha + \beta$ and $y = \alpha - \theta\beta$. We define 
 \begin{equation*}\label{eq:defpsi}
\Psi_{x, y}:= \{Y_{x} \pm Y_{y}, \,\, Z_{x} \pm Z_{y}\}. 
\end{equation*}   
Note that for any imaginary root~$\gamma$, $Y_\gamma = -\sigma_\gamma Y_{-\gamma}$ and $Z_\gamma = \sigma_\gamma Z_{-\gamma}$. Since $x$ and $y$ are imaginary, 
it follows that if $(x, y)\sim (x',y')$, then $\Psi_{x,y} \equiv \Psi_{x',y'}$. 
We now have the following result:

\begin{proposition}\label{prop:s0s0}
Let $(\alpha, \beta)$ be a matched pair and $(x,y)$ be the corresponding {\em  iop}. Then,
$\langle \alpha, \beta + \theta \beta \rangle = 0$ and $c^2_{\alpha, \beta} = c^2_{\alpha, -\theta\beta} > 0$. Moreover,   
\begin{equation*}\label{eq:computes0s0}
[\Phi_\alpha, \Phi_\beta] \equiv
\left\{
\begin{array}{ll}
\Psi_{x, y} & \mbox{if $\Delta$ contain only imaginary roots}, \\
\{Y_\alpha, Z_\alpha, Y_\beta, Z_\beta\} \cup \Psi_{x, y} & \mbox{if $\Delta$ contains complex roots}.
\end{array}
\right.
\end{equation*}
\end{proposition}

\begin{proof}
If $\Delta$ contains only imaginary roots, then the $\alpha$-string $\beta + n\alpha$, for $-q\le n \le p$, is such that $p = q = 1$. Thus, $\langle \alpha, \beta \rangle = \nicefrac{|\alpha|^2(q - p)}{2} = 0$ and  $c^2_{\alpha, \beta} = c^2_{\alpha, -\beta} = 4$. It then follows from~\eqref{lem3:cond11} and\eqref{lem3:cond1} that $$[\mathrm{i}H_\alpha, \Phi_\beta] = [\mathrm{i}H_\beta, \Phi_\alpha] = 0 \quad \mbox{and} \quad [\{Y_\alpha, Z_\alpha \}, \{Y_\beta, Z_\beta\}] \equiv \Psi_{x, y}.$$
We next assume that $\Delta$ contains complex roots. All roots share the same length. We  normalize it to be $\sqrt{2}$. Then, $\langle \alpha, \beta\rangle = -1$ and $\langle \alpha, \theta\beta \rangle = 1$. Thus, $\langle \alpha, \beta + \theta\beta\rangle = 0$ and $\langle \alpha, \beta - \theta \beta\rangle = -2$.   
It follows from~\eqref{lem3:cond11}  that 
$$
\left\{
\begin{array}{lll}
{[\mathrm{i}(H_\alpha + H_{\theta\alpha}), \Phi_\beta] } = 0, & & [\mathrm{i}(H_\beta + H_{\theta\beta}), \Phi_\alpha] = 0,  \\
{[H_\alpha - H_{\theta\alpha}, \Phi_\beta] }\equiv \{Y_\beta, Z_\beta\}, & & [H_\beta - H_{\theta\beta}, \Phi_\alpha] \equiv \{Y_\alpha, Z_\alpha\}. 
\end{array}
\right.
$$
Since neither $\alpha - \beta$ nor $\alpha + \theta\beta$ is a root,  $c^2_{\alpha, \beta} = c^2_{\alpha,-\theta\beta} = 1$. It follows from~\eqref{lem3:cond1} that $[\{Y_\alpha, Z_\alpha \}, \{Y_\beta, Z_\beta\}] \equiv \Psi_{x, y}$. 	
\end{proof}

The above result then implies that $S_1 \equiv S_0 \cup \bigcup_{(x, y)}\Psi_{x, y}$, where the union is over all the {\em iop}s $(x,y)$. If there does not exist an {\em iop}, then $S_1 \equiv S_0$. By Theorem~\ref{prop:matchedpairs}, this is the case if $\Delta$ contains only imaginary roots and the underlying root system is $A_{n}$, $D_n$, or $E_m$ for $m = 6,7,8$. This establishes item~1 of Theorem~\ref{thm:allcases}.

\subsection{Computation of $S_2$}\label{ssec:computes1}

To compute $S_2 = [S_1,S_1]$, it suffices to compute $[\Psi_{x,y}, \Phi_\gamma]$ for $(x, y)$ an {\em  iop} and $\gamma$ an arbitrary root, and $[\Psi_{x, y}, \Psi_{x',y'}]$ for two {\em  iop}s $(x,y)$ and $(x',y')$.

\paragraph{Computation of \text{$[\Psi_{x,y}, \Phi_\gamma]$}.} 
Let $(\alpha, \beta)$ be a matched pair and $(x,y)$ be the corresponding {\em  iop}. 
We define
$
P := R_{x, y} + \{\gamma, -\theta\gamma\} 
$. 
Note that $P$ is closed under $-\theta$. 
Let $T_{\alpha, \beta}:= \pm\{\alpha, \theta\alpha, \beta, \theta\beta\}$. 
We have the following fact:

\begin{proposition}\label{prop:onepaironeroot}
If $P$ does not contain any root or zero, then $[\Psi_{x, y}, \Phi_\gamma] = 0$. 
If $P$ contains zero, then $\gamma\in R_{x, y}$ and $[\Psi_{x, y}, \Phi_\gamma] \equiv \Phi_\gamma$. If $P$ contains at least a root, then the following hold: 
\begin{enumerate}
	\item If $\gamma\in T_{\alpha, \beta}$, then $P$ has two roots $\{\delta, -\theta\delta\}$ and 
			$
			[\Psi_{x, y}, \Phi_\gamma] \equiv \Psi_{x, y} \cup \{Y_{\delta}, Z_\delta\}
			$. 
	\item If $\gamma\in \Delta\backslash\Delta_{\rm sin}$ and $\gamma\not\in T_{\alpha, \beta}$, then $P$ has two roots $\{\delta, -\theta\delta\}$ and $\gamma$ is strongly orthogonal to either $x$ or $y$. Let $u\in \{x, y\}$ be the root such that $\langle \gamma, u\rangle = 0$  and $\{v\}:=\{x,y\}\backslash \{u\}$. Then,  $[\Psi_{x, y}, \Phi_\gamma]\equiv \{Y_{v},Z_{v}, Y_{\delta}, Z_\delta\}$. 
	\item If $\gamma \in \Delta_{\rm sin}$, then $P$ has four roots $R_{x',y'}$ where $(x',y')$ is an {\em  iop} and $[\Psi_{x, y}, \Phi_\gamma] \equiv \Psi_{x, y} \cup \Psi_{x',y'}$.
\end{enumerate}
\end{proposition}

We provide a proof in Appendix~A.

\paragraph{Computation of \text{$[\Psi_{x, y}, \Psi_{x', y'}]$.}}
Let $(\alpha, \beta)$, $(\alpha',\beta')$ be matched pairs and $(x,y)$, $(x',y')$ be the corresponding {\em iop}s.
 Recall that $Q = R_{x, y} + R_{x', y'}$. If $Q$ does not contain  any root or zero, then $[\Psi_{x, y}, \Psi_{x', y'}] = 0$. We assume in the sequel that $Q$ contains either zero or at least a root. We first have the following result:

\begin{proposition}\label{prop:compzeroQ}
	Suppose that $Q$ contains zero; then, there are two cases: 
	\begin{enumerate}
	\item If $R_{x, y} = R_{x', y'}$, then 
	$
	[\Psi_{x, y},\Psi_{x', y'}] \equiv \{\mathrm{i}(H_\alpha + H_{\theta\alpha}), \mathrm{i}(H_\beta + H_{\theta\beta})\}
	$.
	\item If $R_{x, y} \cap R_{x', y'} = \{\pm z\}$ for some root $z$, then 
	$[\Psi_{x, y},\Psi_{x', y'}] \equiv \{\mathrm{i} H_z\}$.
	\end{enumerate}
\end{proposition}

\begin{proof}
For item~1, we note that by Corollary~\ref{cor:orthogonalpair}, $x$ and $y$ are strongly orthogonal.  
It follows from~\eqref{eq:introHHHH} that $[\Psi_{x, y}, \Psi_{x,y}] \equiv \{\mathrm{i}H_{\alpha + \beta} \pm \mathrm{i}H_{\alpha - \theta \beta} \}$. By the linearity $h_{\alpha + \beta} = h_\alpha + h_\beta$ and the scaling $H_\alpha = \nicefrac{2h_\alpha}{|\alpha|^2}$, we obtain that 
$$
\left\{
\begin{array}{lll}
\mathrm{i}H_{\alpha + \beta} & = & \nicefrac{|\alpha|^2}{|\alpha + \beta|^2}\, \mathrm{i}H_\alpha + \nicefrac{|\beta|^2}{|\alpha + \beta|^2}\, \mathrm{i} H_\beta, \vspace{3pt}\\ 
\mathrm{i}H_{\alpha -\theta \beta} & = & \nicefrac{|\alpha|^2}{|\alpha - \theta\beta|^2}\, \mathrm{i}H_\alpha - \nicefrac{|\beta|^2}{|\alpha - \theta\beta|^2}\, \mathrm{i} H_{\theta\beta}, 
\end{array}
\right.
$$ 
Since the root system is not $G_2$, we obtain by Theorem~\ref{prop:matchedpairs} that $|\alpha|^2 = |\beta|^2$ and $|\alpha + \beta|^2 = |\alpha - \theta\beta|^2$. Also, note that $\theta\alpha + \theta \beta = \alpha + \beta$. 
Thus, $$[\Psi_{\alpha, \beta}, \Psi_{\alpha',\beta'}] \equiv \{\mathrm{i}H_{\alpha + \beta} \pm \mathrm{i}H_{\alpha - \theta \beta} \} \equiv \{\mathrm{i}(H_{\alpha} + H_{\theta \alpha}),\mathrm{i}(H_{\beta} + H_{\theta \beta})\}.$$

For item~2, we note that by Prop.~\ref{prop:zeroQ}, any two nonproportional roots out of $R_{x,y} \cup R_{x',y'}$ are strongly orthogonal, so $[\Psi_{x,y}, \Psi_{x',y'}]\equiv [Y_{z}, Z_{z}] \equiv \{\mathrm{i}H_{z}\}$.    	
\end{proof}

We next assume that $Q$ contains at least a root. To proceed, 
we need to introduce a few subsets of $\g_0$, which will be used for the case where the underlying root system is exceptional (i.e., $F_4$ or $E_6$ by Prop.~\ref{prop:f4e6q}). 

Recall that by Prop.~\ref{prop:isometry}, $\Delta^{F_4}_{\rm sin} \approx \Delta^{E_6}_{\rm sin}$ and we will simply write $\Delta_{\rm sin}$. 
Let the decomposition $\Delta_{\rm sin} = \cup^3_{i = 1}\Delta_{{\rm sin}_i}$ be given by~\eqref{eq:singleroots}. We write explicitly $\Delta_{{\rm sin}_i}= \pm \{u_i, v_i, s_i, t_i\}$ and define 
$$
\Sigma_i:= \{Y_{u_i} \pm Y_{v_i} \pm Y_{s_i} \pm Y_{t_i}, \,\, Z_{u_i} \pm Z_{v_i} \pm Z_{s_i} \pm Z_{t_i} \}, \quad \forall i = 1,2,3.
$$  
Note that   
$\Sigma_i$ differs up to sign if another set of roots $\{u_i, v_i, s_i, t_i\}$ is used. 

Recall that $
\xi= \langle x, x'\rangle \langle y, y'\rangle \langle x, y'\rangle \langle y, x'\rangle 
$ is defined in~\eqref{eq:defxi}. 
Since $Q$ contains at least a root, it follows from Prop.~\ref{prop:f4e6q} that $\xi\neq 0$. 
Consequently, we have the following result:

\begin{proposition}\label{prop:rootsQ}
	Suppose that $Q$ has at least a root; then, there are two cases:
	\begin{enumerate}
	\item If $\xi > 0$, then there is an iop $(x'',y'')$ such that 
	$[\Psi_{x, y}, \Psi_{x',y'}] \equiv \Psi_{x'',y''}$.
	\item If $\xi < 0$, then the underlying root system can only be $F_4$ or $E_6$. Let $i,j\in \{1,2,3\}$ be such that $R_{x, y}\subset \Delta_{{\rm sin}_i}$ and $R_{x',y'}\subset \Delta_{{\rm sin}_j}$. Then, $i \neq j$ and, moreover,   
$[\Psi_{x, y}, \Psi_{x',y'}] \equiv \Sigma_k$ where $\{k\} := \{1,2,3\}\backslash \{i,j\}$.
	\end{enumerate}
\end{proposition}

We provide a proof in Appendix~B. 

It follows from  
Propositions~\ref{prop:onepaironeroot}, ~\ref{prop:compzeroQ} and~\ref{prop:rootsQ} that  
if the underlying root system is classical, then $S_2 \equiv S_1$. Item~2 of Theorem~\ref{thm:allcases} is then established. However, if the underlying root system is  $F_4$ or $E_6$, then $S_2 \equiv S_1 \cup \bigcup^3_{i = 1}\Sigma_i$.

\subsection{Computation of $S_3$: On root systems $F_4$ and $E_6$}\label{ssec:computes3}

We consider here two cases: Either the root system is $F_4$ and $\Delta$ contains only imaginary roots, or, the root system is $E_6$ and $\Delta$ contains complex roots.  
To compute $S_3 = [S_2, S_2]$, it suffices to fix an $i= 1,2,3$ and compute the following: 
(1) $[\Sigma_i, \Phi_\alpha]$ for an arbitrary root $\alpha$, (2) $[\Sigma_i, \Psi_{x, y}]$ for an {\em iop} $(x, y)$, and (3) $[\Sigma_i, \Sigma_j]$ for $j = 1,2,3$. 
The computation is, in fact, on the level of single roots. 
By Prop.~\ref{prop:f4e6q}, $\Delta^{F_4}_{\rm sin}$ and $\Delta^{E_6}_{\rm sin}$ are essentially the same.

First, for $[\Sigma_i, \Phi_\alpha]$, we have the following results:

\begin{proposition}\label{prop:computesigmaphi}
The following hold for $[\Sigma_i, \Phi_\alpha]$:
\begin{enumerate}
\item   If $\alpha\in \Delta_{{\rm sin}_i}$, then $[\Sigma_i, \Phi_\alpha] \equiv \Phi_\alpha$.
\item 	If $\alpha\in \Delta_{{\rm sin}_j}$ for $j\neq i$, then $[\Sigma_i, \Phi_\alpha]\equiv \Sigma_i \cup \Sigma_k$ where $\{k\}:= \{1,2,3\} \backslash \{i,j\}$.
\item 	If  $\alpha\in \Delta \backslash \Delta_{\rm sin}$, then there is another root $\beta\in \Delta \backslash \Delta_{\rm sin}$, unique up to $-\theta$, such that $(\alpha, \beta)$ is a matched pair. Moreover, if we let $(x,y)$ be the corresponding {\em iop}, then $[\Sigma_i, \Phi_\alpha] \equiv \Psi_{x,y}\cup \{Y_{-\beta}, Z_{-\beta}\}$.
\end{enumerate}
\end{proposition}

Next, for $[\Sigma_i, \Psi_{x, y}]$, we have the following results:

\begin{proposition}\label{prop:computesigmapsi} 
Let $(\alpha, \beta)$ be a matched pair and $(x,y)$ be the corresponding {\em iop}. Then, the following hold for $[\Sigma_i, \Psi_{x, y}]$:
\begin{enumerate} 
\item If $R_{x,y} \subset \Delta_{{\rm sin}_i}$, then
$[\Sigma_i, \Psi_{x, y}] \equiv \{\mathrm{i}(H_\alpha + H_{\theta\alpha}), \mathrm{i}(H_\beta + H_{\theta\beta})\}$.
 \item If $R_{x, y} \subset \Delta_{{\rm sin}_j}$ for $j\neq i$, then there are {\em iop}s $(x', y')$ and $(x'',y'')$ such that $R_{x',y'} \cup R_{x'',y''} = \Delta_{{\rm sin}_k}$ for $\{k\} := \{1,2,3\} \backslash \{i,j\}$ and, moreover, 
 $[\Sigma_i, \Psi_{x,y}]\equiv \Sigma_k \cup \Psi_{x', y'} \cup \Psi_{x'', y''}$.  
  \end{enumerate}
\end{proposition}

Finally, for $[\Sigma_i, \Sigma_j]$, we have the following results:

\begin{proposition}\label{prop:computesigmasigma} 
The following hold for $[\Sigma_i,\Sigma_j]$:
\begin{enumerate}
\item If $j = i$, then 
 $[\Sigma_i,\Sigma_j] \equiv \{\mathrm{i} H_x \mid x\in \Delta_{\rm sin} \backslash \Delta_{{\rm sin}_i} \}$.

\item If $j\neq i$, then $[\Sigma_i, \Sigma_j] \equiv \Sigma_k \cup \bigcup_{x\in \Delta_{{\rm sin}_k}}\{Y_x, Z_x\}$ for $\{k\}:= \{1,2,3\} \backslash \{i,j\}$. 
\end{enumerate}
\end{proposition}

We provide proofs of the above results in Appendices C, D, and E, respectively. 
The above results imply that $S_3 \equiv S_2$ and item~3 of Theorem~\ref{thm:allcases} is established.

\subsection{Computation of $[S^*, S^*]$: On root system $G_2$}\label{ssec:caseG2}
Note that for the special case, $\g_0$ is the compact real form of $\g$ (since $\g_0$ is neither complex nor a split real form of $\g$).  
For each $i = 1,2,3$, let $(i,j,k)$ be a cyclic rotation of $(1,2,3)$ and define
$x_i := e_j - e_k$ and $y_i:= e_j + e_k - 2e_i$. 
Note that by item~({\em g}) of Theorem~\ref{prop:orthogonalpairs}, each $(x_i,y_i)$ is an {\em iop}. 
For ease of notation, we let 
$\Phi_i:=\Phi_{y_i}$ and $\Psi_i: =\{Y_{y_i} \pm Y_{x_i}, Z_{y_i} \pm Z_{x_i}\}$ for each~$i = 1,2,3$.   
Then, the set $S^*$ defined in~\eqref{eq:defsstar} can be re-written as 
$
S^* =  \cup^3_{i = 1}(\Phi_i \cup \Psi_i)
$. 

Note that  every element in $S_0$ can be expressed as a linear combination of elements in $S^*$, so $S^*$ spans~$\g_0$. 
We have the following computational result: 

\begin{proposition}\label{prop:computeg2} 
The following hold for $[S^*, S^*]$:
\begin{enumerate}
\item For any given $i\in \{1,2,3\}$, let $\{j,k\} := \{1,2,3\} \backslash \{i\}$. Then,  
$[\Phi_i,\Psi_i] \equiv [\Phi_i, \Phi_i] \equiv \Phi_i$ and $ 
[\Psi_i, \Psi_i]\equiv \{\mathrm{i} H_{y_j},  \mathrm{i}H_{y_k}\}$.  
\item For any given $(i,j)$ with $i\neq j$, let $\{k\} := \{1,2,3\} \backslash \{i,j\}$. Then, 
$[\Phi_i, \Phi_j] \equiv \cup^3_{l = 1} \{Y_{y_l}, Z_{y_l}\}$,  
$[\Phi_i,\Psi_j] \equiv \Psi_j \cup \Psi_k$, and $[\Psi_i,\Psi_j]\equiv \Psi_{k} \cup \{Y_{y_k}, Z_{y_k}\}$. 
\end{enumerate}
\end{proposition}

We provide a proof in Appendix~F. 
Item~4 of Theorem~\ref{thm:allcases} then follows from the above result. The proof of Theorem~\ref{thm:allcases} is now complete.

\bibliographystyle{amsplain}%
\bibliography{alias,bib}

\appendix


\section{Proof of Prop.~\ref{prop:onepaironeroot}}
It should be clear that $P$ contains zero if and only if $\gamma\in R_{x, y}$.  Because the underlying root system is not $G_2$, by Corollary~\ref{cor:orthogonalpair} and Theorem~\ref{prop:orthogonalpairs}, $x$ and $y$ are single, imaginary roots and are strongly orthogonal to each other. Thus, $[\Psi_{x, y}, \Phi_\gamma] \equiv [\Phi_\gamma, \Phi_\gamma] \equiv \Phi_\gamma$.  
We now assume that $P$ contains at least a root and establish the three items subsequently:
\vspace{.1cm}

{\em Proof of item~1.}  
We assume, without loss of generality, that $\gamma = \alpha$. Then, $\{-\beta, \theta\beta\}$ are the roots contained in $P$. By Prop.~\ref{prop:s0s0}, $\langle\alpha, \beta + \theta \beta\rangle = 0$. Thus, $\langle \alpha, \alpha + \beta \rangle = \langle \alpha, \alpha - \theta\beta \rangle > 0$. It follows from~\eqref{lem3:cond11} that  
$[\Psi_{x, y}, \mathrm{i}(H_\alpha + H_{\theta \alpha})]  \equiv \Psi_{x, y}$. 
Next, note that $c_{\alpha + \beta, \alpha} = c_{\alpha - \theta\beta, \alpha} = 0$ and $c^2_{\alpha + \beta, -\theta\alpha} = c^2_{\alpha - \theta\beta, -\theta\alpha} = 1$.  
It then follows from~\eqref{lem3:cond1} that $[\Psi_{x, y}, \{Y_\alpha, Z_\alpha\}] \equiv \{Y_{-\beta}, Z_{-\beta}\}$.  
\vspace{.1cm}

{\em Proof of item~2.} We assume, without loss of generality, that $\delta :=\alpha + \beta + \gamma$ is a root in~$P$. Then,  $-\theta\delta = -(\alpha + \beta) - \theta \gamma$ is a root as well. We need to show that $\delta$ and $-\theta\delta$ are the only two roots in $P$. First, note that by Theorem~\ref{prop:matchedpairs}, $(\alpha + \beta)\in \Delta_{\rm sin}$, so $(\alpha + \beta, \gamma)$ and $(\alpha + \beta, -\gamma)$ are not matched pairs. Since $\delta$ and $-\theta\delta$ are roots, the two elements $-(\alpha + \beta) + \gamma$ and $(\alpha + \beta) - \theta \gamma$ of $P$ cannot be roots. 
For the other four elements $\{\pm(\alpha - \theta\beta) + \gamma, \pm(\alpha - \theta\beta) -\theta\gamma\}$ of $P$, it suffices to show that $\alpha - \theta\beta$ is strongly orthogonal to $\gamma$ (note that since $\delta = \alpha + \beta + \gamma$ is a root under the assumption, $\gamma$ cannot be orthogonal to $\alpha + \beta$, i.e., $u= \alpha -\theta\beta$ and $v= \alpha + \beta$ in the case). In fact, 
since $\alpha - \theta \beta$ is a single (and imaginary) root, we only need to show that $\alpha - \theta\beta$ is orthogonal to $\gamma$. Orthogonality of the two roots implies strong orthogonality.  

To establish orthogonality of the two roots, we consider two cases: If $\Delta$ contains only imaginary roots, then $\gamma\in \Delta_{\rm sh}$ and $\alpha + \beta\in\Delta_{\rm long}$. Thus, 
$
\nicefrac{\langle \alpha + \beta, \gamma\rangle}{|\gamma|^2} = \nicefrac{2\langle \alpha + \beta, \gamma\rangle}{|\alpha + \beta|^2} = -1$,  
which holds if and only if $\nicefrac{\langle \alpha, \gamma\rangle}{|\gamma|^2} = \nicefrac{\langle  \beta, \gamma\rangle}{|\gamma|^2} = -\nicefrac{1}{2}$. It then follows that $\langle\alpha - \beta,\gamma \rangle = 0$. 

We now assume that $\Delta$ contains complex roots. In this case, $\gamma\in \Delta_{\rm comp}$ and $\alpha + \beta\in \Delta_{\rm im}$. We normalize the common length of a root in $\Delta$ to be $\sqrt{2}$. 
Since $\alpha + \beta + \gamma$ is a (complex) root,  
$\langle \alpha + \beta, \gamma \rangle = -1$. We assume, without loss of generality, that $\langle \alpha, \gamma\rangle = -1$ and, hence, $\alpha + \gamma$ is a root. Now, suppose for the moment that $\langle \theta \beta, \gamma\rangle \ge 0$; then, $\langle \alpha - \theta\beta, \gamma\rangle< 0$ and, hence,  $(\alpha + \gamma) - \theta \beta$ is a root. It then follows that $(\alpha + \gamma, \beta)$ is a matched pair, which contradicts the fact that $\alpha + \beta + \gamma$ is a complex root. Thus, we must have that $\langle \theta\beta, \gamma\rangle < 0$. 
Furthermore, since $\gamma\not\in T_{\alpha,\beta}$, we obtain that $\langle \theta\beta, \gamma\rangle = -1$. It then follows that $\langle\gamma, \alpha - \theta\beta\rangle = -1 - (-1) = 0$. 

To compute $[\Psi_{x, y},\Phi_\gamma]$, we first use~\eqref{lem3:cond11} and the fact that $\langle \gamma, \alpha -\theta\beta \rangle = 0$ to obtain that $[\Psi_{x, y}, \mathrm{i} H_\gamma] \equiv \{Y_{\alpha + \beta}, Z_{\alpha + \beta}\}$. We next note that $c^2_{\alpha + \beta, \gamma} = 1$ and $c_{\alpha + \beta, -\theta\gamma} = c_{\alpha - \theta\beta, \gamma}  = c_{\alpha - \theta\beta, -\theta\gamma} = 0$. Thus, by~\eqref{lem3:cond1}, we obtain that $[\Psi_{x, y}, \{Y_\gamma, Z_\gamma\}] \equiv \{Y_{\delta}, Z_{\delta}\}$. 
\vspace{.1cm}

{\em Proof of item~3.} We assume again that $\alpha + \beta + \gamma$ is a root in~$P$. By Theorem~\ref{prop:matchedpairs}, $|\gamma|^2 = |\alpha + \beta|^2\ge |\alpha|^2 = |\beta|^2$. Thus,  $\nicefrac{\langle \alpha + \beta, \gamma \rangle}{|\gamma|^2} = -\nicefrac{1}{2}$, so either $\langle \alpha, \gamma\rangle$ or $\langle\beta, \gamma\rangle$ is zero. Without loss of generality, we assume that $\langle \alpha, \gamma\rangle = 0$ and $\nicefrac{\langle\beta, \gamma\rangle}{|\gamma|^2} = -\nicefrac{1}{2}$. Next, note that $\gamma$ is imaginary and, hence, $\langle \theta \alpha, \gamma \rangle = \langle \alpha, \gamma \rangle = 0$. Then, $\langle\beta -\theta\alpha, \gamma\rangle = \langle\beta, \gamma\rangle < 0$,  so $\beta  + \gamma - \theta\alpha$ is a root. It follows that $(\alpha, \beta + \gamma)$ is a matched pair. Let $(x',y')$ be the corresponding {\em iop}. Then, the four roots of $R_{x', y'}$ are contained in $P$ (note that $\theta\gamma = \gamma$).  The other four elements in $P$ cannot be roots. 
To see this, note that $\gamma\in \Delta_{\rm \sin}$ and, hence, $(\alpha + \beta, \gamma)$ and $(\alpha - \theta\beta, \gamma)$ cannot be matched pairs. Since $\alpha + \beta +\gamma$ and $\alpha - \theta\beta  - \gamma$ are roots, $\alpha + \beta - \gamma$ and $\alpha - \theta \beta + \gamma$ (as well as their negatives) are not roots. 

To compute $[\Psi_{x, y}, \Phi_\gamma]$, we first note that 
$\langle \gamma, \alpha + \beta\rangle = -\langle \gamma, \alpha -\theta \beta\rangle = \langle\gamma, \beta\rangle < 0$.  
It then follows from~\eqref{lem3:cond11} that $[\Psi_{x, y}, \mathrm{i} H_\gamma] \equiv \Psi_{x, y}$.  
We next note that $c^2_{\alpha + \beta, \gamma} = c^2_{\alpha - \theta\beta, -\gamma} = 1$ and $c_{\alpha + \beta, -\gamma} = c_{\alpha - \theta\beta, \gamma} = 0$.  
Thus, by~\eqref{lem3:cond1}, $[\Psi_{x, y}, \{Y_\gamma, Z_\gamma\}] \equiv \Psi_{x',y'}$.   
\hfill{\qed}

\section{Proof of Prop.~\ref{prop:rootsQ}}
To establish Prop.~\ref{prop:rootsQ}, we need a few preliminary results. We first have the following fact adapted from~\cite[Prop.~6.104]{AWK:02}: 

\begin{lemma}\label{lem:sigmaxsigmay}
If $(x, y)$ is an {\em iop}, then $\sigma_{x} = \sigma_{y}$. 	
\end{lemma}

We next introduce the following fact which is adapted from Lemmas 5.1 and 5.3 in~\cite[Ch. III]{SH:79}:

\begin{lemma}\label{lem:auxiliary}
Let $s, t, u, v$ be four arbitrary roots, no two of which have sum~$0$. If $s + t + u + v = 0$, then the following hold: 
$$
\frac{c_{s,t}c_{x, y}}{|s + t|^2} + \frac{c_{t,u}c_{s,v}}{|t + u|^2} + \frac{c_{u,s}c_{t,v}}{|u + s|^2} = 0.
$$  	
\end{lemma}

With the Lemmas above, we will now establish Prop.~\ref{prop:rootsQ}. Since the underlying root system is not $G_2$, it follows from Theorem~\ref{prop:matchedpairs} that all roots $x$, $y$, $x'$, and $y'$ share the same length. We normalize the length to be $\sqrt{2}$. Since $R_{x, y}$ does not intersect $R_{x',y'}$ (otherwise, $Q$ contains zero but not any root), each inner-product in the expression~\eqref{eq:defxi} of $\xi$ can only be~$1$ or~$-1$.  
\vspace{.1cm}

{\em Proof of item~1.} 
We assume, without loss of generality, that 
all inner-products in~\eqref{eq:defxi} all $-1$. It follows from computation that 
\begin{equation}\label{eq:lin}
\langle x + x', y + y' \rangle = \langle x + y' , y + x' \rangle = -2 \quad \langle x + x', x + y' \rangle = 0.
\end{equation} 
By Theorem~\ref{prop:matchedpairs}, the length of a single root is greater than or equal to the length of a nonsingle root. Thus, for~\eqref{eq:lin} to hold, we must have that  
$x + x' = - (y + y')$ and $(x + y') = - (y + x')$ are single roots and $\pm \{x + x', x + y'\}$ are the four roots in~$Q$. Furthermore, by Prop.~\ref{prop:f4e6q}, there is an {\em iop} $(x'', y'')$ such that $x'' = x + x'$ and $y'' = x + y'$.  

We now compute $[\Psi_{x, y}, \Psi_{x',y'}]$. Let $\epsilon$ and $\epsilon'$ be free variables that take values from the set $\{1,-1\}$. Then, by~\eqref{lem3:cond1}, 
\begin{equation}\label{eq:comptwopairs}
\left\{
\begin{array}{lllll}
{[Y_{x} +  \epsilon Y_{y}, Y_{x'} + \epsilon'Y_{y'}]} & = & -{[Z_{x} +  \epsilon Z_{y}, Z_{x'} + \epsilon'Z_{y'}]} & = &
\mu_1 Y_{x + x'} + \mu_2 Y_{x + y'}, \\
 
{[Y_{x} +  \epsilon Y_{y}, Z_{x'} + \epsilon'Z_{y'}]} & = & {[Z_{x} +  \epsilon Z_{y}, Y_{x'} + \epsilon'Y_{y'}]} & = & 
\nu_1 Z_{x + x'} 
+ \nu_2 Z_{x + y'}, 
\end{array}
\right.
\end{equation}
where the coefficients $\mu_1$, $\mu_2$, $\nu_1$, and $\nu_2$ are given by
\begin{equation}\label{eq:munu12}
\left\{
\begin{array}{lllllll}
\mu_1 & := & c_{x, x'} -\sigma_{x''} \epsilon\epsilon' c_{y, y'}, & & 
\mu_2 & := & \epsilon' c_{x, y'} - \sigma_{y''} \epsilon c_{y, x'},\\
\nu_1 & := &  c_{x, x'} +\sigma_{x''} \epsilon\epsilon' c_{y, y'}, & &
\nu_2 & := &\epsilon' c_{x, y'} + \sigma_{y''} \epsilon c_{y, x'}.
\end{array}
\right. 
\end{equation}

We show below that for any $\epsilon, \epsilon' \in\{1,-1\}$,  
$\mu^2_1 = \mu^2_2$ and $\nu^2_1 = \nu^2_2$. Note that if the two equalities hold, then by~\eqref{eq:comptwopairs} and~\eqref{eq:munu12}, $[\Psi_{x, y}, \Psi_{x',y'}] \equiv \Psi_{x'',y''}$.  
To establish the equalities, we first note that $c^2_{*,*} = 1$ for any $c_{*,*}$ in~\eqref{eq:munu12}.    
Next, note that by Lemma~\ref{lem:sigmaxsigmay}, $\sigma_{x''} = \sigma_{y''}$.  
It now suffices to show that
$
c_{x, x'} c_{y, y'} =  c_{x, y'}c_{y, x'}. 
$
For that, note that the four roots $x$, $x'$, $y$, and $y'$ satisfy the assumption of Lemma~\ref{lem:auxiliary}, i.e., $x + x' + y + y' = 0$ and no two of which have sum~$0$. Thus, 
$$
\frac{c_{x, x'}c_{y, y'}}{|x + x'|^2} + \frac{c_{x',y}c_{x,y'}}{|x' + y|^2} + \frac{c_{y,x}c_{x',y'}}{|x + y|^2} = 0.
$$ 
Since $x + y$ and $x' + y'$ are not roots, $c_{y,x} = c_{x',y'} = 0$. Further, note that $x''$ and $y''$ are single roots and, hence,  $|x''|^2 = |x + x'|^2 = |x' + y|^2 = |y''|^2$. We thus conclude that   
$c_{x, y}c_{y, y'} = -c_{x', y}c_{x, y'} = c_{y, x'}c_{x, y'}$.  
\vspace{.1cm}

{\em Proof of item~2.} 
	We assume, without loss of generality, that the first three inner-products in~\eqref{eq:defxi} are~$-1$ and the last one is~$1$. Then, 
	$$u: =x + x', \quad v:= y + y', \quad s := x + y', \quad t := y - x'$$
	are roots and, by computation, any two out of the four roots are orthogonal. By item~2 of Prop.~\ref{prop:f4e6q}, the underlying root system is $F_4$ or $E_6$. Moreover, we must have that $R_{x, y}\subset \Delta_{{\rm sin}_i}$ and $R_{x', y'}\subset \Delta_{{\rm sin}_j}$ with $i\neq j$ and $\Delta_{{\rm sin}_k} = \pm\{u, v, s, t\}$. Let $\epsilon, \epsilon'$ be variables that take values in $\{1,-1\}$. Then, by~\eqref{lem3:cond1}, 
	$$
	\left\{
	\begin{array}{lll}
		{[Y_x + \epsilon Y_y, Y_{x'} + \epsilon' Y_{y'}]} & = & 
		c_{x,x'} Y_{u} + \epsilon \epsilon' c_{y,y'} Y_{v} + \epsilon' c_{x,y'} Y_{s} - \epsilon \sigma_{x'} c_{y, -x'} Y_{t}, \vspace{3pt}\\
		{[Z_x + \epsilon Z_y, Z_{x'} + \epsilon' Z_{y'}]} & = & 
		-c_{x,x'} Y_{u} - \epsilon \epsilon' c_{y,y'} Y_{v} - \epsilon' c_{x,y'} Y_{s} - \epsilon \sigma_{x'} c_{y, -x'} Y_{t}, \vspace{3pt}\\
		{[Y_x + \epsilon Y_y, Z_{x'} + \epsilon' Z_{y'}]} & = & 
		c_{x,x'} Z_{u} + \epsilon \epsilon' c_{y,y'} Z_{v} + \epsilon' c_{x,y'} Y_{s} + \epsilon \sigma_{x'} c_{y, -x'} Z_{t}, \vspace{3pt}\\
		{[Z_x + \epsilon Z_y, Y_{x'} + \epsilon' Y_{y'}]} & = & 
		c_{x,x'} Z_{u} + \epsilon \epsilon' c_{y,y'} Z_{v} + \epsilon' c_{x,y'} Z_{s} - \epsilon \sigma_{x'} c_{y, -x'} Z_{t}. 
	\end{array}
	\right.
	$$
	Note that $c^2_{*,*} = 1$ for any $c_{*,*}$ in the above expression. By taking different values of $\epsilon, \epsilon'\in \{1,-1\}$, we obtain that $[\Psi_{x, y}, \Psi_{x',y'}] \equiv \Sigma_k$.
\hfill{\qed}

\section{Proof of Prop.~\ref{prop:computesigmaphi}}
We establish the three items subsequently.
\vspace{.1cm}

{\em Proof of item~1.} 
Note that all roots in $\Delta_{{\rm sin}_i}$ are either proportional or orthogonal. Thus, if $\alpha \in \Delta_{{\rm sin}_i}$, then $[\Sigma_i, \Phi_\alpha] \equiv [\Phi_\alpha,\Phi_\alpha] \equiv \Phi_\alpha$. 
\vspace{.1cm}

{\em Proof of item~2.} 
We now assume that $\alpha \in \Delta_{{\rm sin}_j}$ with $j\neq i$. First, note that $\langle \alpha, x\rangle^2 = \langle \alpha, y\rangle^2$ for any $x, y\in \Delta_{{\rm sin}_i}$. Thus, $[\Sigma_i, \mathrm{i}H_\alpha]\equiv \Sigma_i$. We next note that the roots contained in the set $\Delta_{{\rm sin}_i} \pm \{\alpha\}$ are given by $\Delta_{{\rm sin}_k}$. It follows from computation that $[\Sigma_i, \{Y_\alpha, Z_\alpha\}] \equiv \Sigma_k$. 
\vspace{.1cm}

{\em Proof of item~3.} 
For convenience, we let $\Delta_{\rm ns}:=\Delta\backslash\Delta_{\rm sin}$ be the set of nonsingle roots. Recall that for either case $F_4$ or $E_6$, we have decomposed  $\Delta_{\rm ns} = \cup^3_{l = 1} \Delta_{{\rm ns}_l}$  (see~\eqref{eq:F4shortroots} and~\eqref{eq:E6comproots}) in a way such that if $(\alpha, \beta)$ is a matched pair, then $\alpha, \beta$ belong to the same $\Delta_{{\rm ns}_l}$. Examination of each $\Delta_{{\rm ns}_l}$ (for both cases) leads to the following fact: For each $\alpha\in \Delta_{{\rm ns}_l}$, there exist three different subsets $\{\beta_{i}, -\theta\beta_{i}\}$ for $i = 1,2,3$ of $\Delta_{{\rm ns}_l}$ such that $(\alpha, \beta_{i})$ is a matched pair. Moreover, if we let $(x_i,y_i)$ be the {\em iop} corresponding  to $(\alpha, \beta_i)$, then $R_{x_i, y_i} \subset \Delta_{{\rm sin}_{i}}$. Thus, for a given $i\in \{1,2,3\}$ and a given root $\alpha\in \Delta_{{\rm ns}_l}$, there is a unique root $\beta \in \Delta_{{\rm ns}_l}$   
(unique up to $-\theta$) such that $(\alpha, \beta)$ is a matched pair and $\Delta_{{\rm sin}_i}$ contains $R_{x,y}$ as roots where $(x,y)$ is the {\em iop} corresponding to $(\alpha, \beta)$. It follows that $\{-\beta, \theta\beta\}$ are the roots contained in the set $\Delta_{{\rm sin}_i} + \{\alpha, -\theta \alpha\}$. In particular, $\alpha$ is strongly orthogonal to any root out of $\Delta_{{\rm sin}_i} \backslash R_{x,y}$. Then, by item~1 of Prop.~\ref{prop:onepaironeroot},  
 $[\Sigma_i, \Phi_\alpha] \equiv [\Psi_{x, y}, \Phi_\alpha] \equiv \Psi_{x,y}\cup \{Y_{-\beta}, Z_{-\beta}\}$. 
\hfill{\qed}

\section{Proof of Prop.~\ref{prop:computesigmapsi}}
We establish below the two items.
\vspace{.1cm}

{\em Proof of item~1.}
Since $R_{x, y} \subset \Delta_{{\rm sin}_i}$, roots in $R_{x, y}$ are strongly orthogonal to roots in $\Delta_{{\rm sin}_i} \backslash R_{x, y}$. Thus, $[\Sigma_i, \Psi_{x, y}] \equiv [\Psi_{x, y}, \Psi_{x, y}]$ and, by item 1 of Prop.~\ref{prop:compzeroQ}, $[\Psi_{x, y}, \Psi_{x, y}] \equiv \{\mathrm{i}(H_\alpha + H_{\theta\alpha}), \mathrm{i}(H_\beta + H_{\theta\beta})\}$.
\vspace{.1cm}

{\em Proof of item~2.}
Because the long roots in $F_4$ and the imaginary roots in $E_6$ are related by a linear isometry (Prop.~\ref{prop:isometry}) and because the analysis is on the level of single roots, it suffices to consider the case where the underlying root system is $F_4$. Furthermore, by symmetry, we can assume, without loss of generality, that $i = 1$ and $R_{x, y} = \{\pm e_2 \pm e_3\}\subset \Delta_{{\rm sin}_2}$.  

To prove the item, we need to have a few preliminary results. For ease of notation, we let  
\begin{equation}\label{eq:defuvst}
\left\{
\begin{array}{llll}
	u_1 := e_1 + e_2, & v_1 := e_1 - e_2, & s_1 := e_3 + e_4, & t_1 := e_3 - e_4, \\
	u_2 := e_3 - e_2, & v_2 := e_2 + e_3, & s_2 := e_1 - e_4, & t_2 := e_1 + e_4, \\
	u_3 := e_1 + e_3, & v_3 := e_1 - e_3, & s_3 := e_2 + e_4, & t_3 := e_2 - e_4. 
\end{array}
\right.
\end{equation} 
Note that by~\eqref{eq:singleroots}, $\Delta_{{\rm sin}_i} = \pm \{u_i, v_i, s_i, t_i\}$ for any $i = 1,2,3$ . 
%
%
We first have the following fact:   

\begin{lemma}\label{lem:defsigmai}
For any $i = 1,2,3$, $\sigma_{u_i} = \sigma_{v_i} = \sigma_{s_i} = \sigma_{t_i}=:\sigma_i$. 	
\end{lemma}

\begin{proof}
	By Theorem~\ref{prop:orthogonalpairs}, any two orthogonal roots in $\Delta_{\rm sin}$ form an {\em iop}. Since $u_i$, $v_i$, $s_i$, and $t_i$ are pairwise orthogonal, the result is then a consequence of Lemma~\ref{lem:sigmaxsigmay}. 
\end{proof}

Given an $i = 1,2,3$, we let $\epsilon_i := (\epsilon_{v_i}, \epsilon_{s_i}, \epsilon_{t_i})$ b a vector in $\R^3$ whose value will be specified later. 
We define elements of $\g_0$ as follows: 
\begin{equation}\label{eq:defsigmayz}
\left\{
\begin{array}{lll}
\mathbf{Y}_i(\epsilon_i) & := & Y_{u_i} + \epsilon_{v_i} Y_{v_i} + \epsilon_{s_i} Y_{s_i} + \epsilon_{t_i} Y_{t_i}, \vspace{3pt}\\
\mathbf{Z}_i(\epsilon_i) & := & Z_{u_i} + \epsilon_{v_i} Z_{v_i} + \epsilon_{s_i} Z_{s_i} + \epsilon_{t_i} Z_{t_i}.
\end{array}
\right.
\end{equation}
The following result directly follows from~\eqref{lem3:cond1},~\eqref{eq:defuvst}, and Lemma~\eqref{lem:defsigmai}: 

\begin{lemma}\label{lem:computesigmasigma}  
The following computational results hold:   
\begin{equation}\label{eq:sigmasigma}
\left\{
\begin{array}{lll}
	[\mathbf{Y}_1(\epsilon_1), \mathbf{Y}_2(\epsilon_2)] & = & 
	\mu_1 Y_{u_3} -  \mu_2 Y_{v_3} - \mu_3 Y_{s_3} - \mu_4 Y_{t_3}, \vspace{3pt}\\
	
	[\mathbf{Z}_1(\epsilon_1), \mathbf{Z}_2(\epsilon_2) ] & = & 
	-\mu_1 Y_{u_3} -  \mu_2 Y_{v_3} - \mu_3 Y_{s_3} - \mu_4 Y_{t_3}, \vspace{3pt}\\
		
	[\mathbf{Y}_1(\epsilon_1), \mathbf{Z}_2(\epsilon_2) ] & = & 
	\nu_1 Z_{u_3} + \nu_2 Z_{v_3} + \nu_3 Z_{s_3} + \nu_4 Z_{t_3}, \vspace{3pt}\\
	
	[\mathbf{Z}_1(\epsilon_1), \mathbf{Y}_2(\epsilon_2) ] & = &
	\nu_1 Z_{u_3} -  \nu_2 Z_{v_3} - \nu_3 Z_{s_3} - \nu_4 Z_{t_3}, \vspace{3pt}\\	
\end{array}
\right. 
\end{equation}
where the coefficients $\mu_i$ and $\nu_i$, for $i= 1,2,3,4$, are given by
 \begin{equation}\label{eq:c**}
 \left\{
\begin{array}{lll}
	\mu_1 & := & \nu_1 \,\, := \,\, c_{u_1,u_2} + \epsilon_{v_1}\epsilon_{v_2} c_{v_1,v_2} + \epsilon_{s_1}\epsilon_{s_2} c_{s_1,s_2} + \epsilon_{t_1}\epsilon_{t_2} c_{t_1,t_2}, \vspace{3pt}\\
	\mu_2 & := & \sigma_2 \epsilon_{v_2} c_{u_1, -v_2} + \sigma_2 \epsilon_{v_1} c_{v_1, -u_2} - \epsilon_{s_1} \epsilon_{t_2} c_{s_1, -t_2} - \epsilon_{t_1} \epsilon_{s_2} c_{t_1, -s_2}, \vspace{3pt}\\
	\mu_3 & := & \sigma_2 \epsilon_{s_2} c_{u_1, -s_2} - \epsilon_{v_1}\epsilon_{t_2} c_{v_1, -t_2} + \sigma_2 \epsilon_{s_1} c_{s_1, -u_2} - \epsilon_{t_1}\epsilon_{v_2} c_{t_1, -v_2},\vspace{3pt}\\ 
	\mu_4 & := & \sigma_2 \epsilon_{t_2} c_{u_1,-t_2} - \epsilon_{v_1}\epsilon_{s_2} c_{v_1, -s_2} - \epsilon_{s_1}\epsilon_{v_2} c_{s_1,-v_2} + \sigma_2\epsilon_{t_1} c_{t_1,-u_2}, \vspace{3pt}\\
	\nu_2 & := & \sigma_2 \epsilon_{v_2} c_{u_1, -v_2} + \sigma_2 \epsilon_{v_1} c_{v_1, -u_2} + \epsilon_{s_1} \epsilon_{t_2} c_{s_1, -t_2} + \epsilon_{t_1} \epsilon_{s_2} c_{t_1, -s_2}, \vspace{3pt}\\
	\nu_3 & := & \sigma_2 \epsilon_{s_2} c_{u_1, -s_2} + \epsilon_{v_1}\epsilon_{t_2} c_{v_1, -t_2} + \sigma_2 \epsilon_{s_1} c_{s_1, -u_2} + \epsilon_{t_1}\epsilon_{v_2} c_{t_1, -v_2},\vspace{3pt}\\ 
	\nu_4 & := & \sigma_2 \epsilon_{t_2} c_{u_1,-t_2} + \epsilon_{v_1}\epsilon_{s_2} c_{v_1, -s_2} + \epsilon_{s_1}\epsilon_{v_2} c_{s_1,-v_2} + \sigma_2\epsilon_{t_1} c_{t_1,-u_2}. 
\end{array}
\right.
\end{equation}
\end{lemma}

By item~2 of Lemma~\ref{theorem:rootspacedecomp}, $c^2_{*,*} = 1$ for any $c_{*,*}$ in the expression~\eqref{eq:c**}. Furthermore, we have the following fact:

\begin{lemma}\label{lem:equalitiesforc}
The following equalities hold for the $c_{*,*}$ in~\eqref{eq:c**}:
\begin{equation*}
\left\{
\begin{array}{lllllll}
c_{u_1,u_2} c_{v_1,v_2} & = & c_{u_1,-v_2}c_{v_1,-u_2}, & \quad 
& c_{u_1,u_2}c_{s_1,s_2} & = & c_{u_1,-s_2}c_{s_1,-u_2}, \\
c_{u_1,u_2} c_{t_1,t_2} & = & c_{u_1,-t_2}c_{t_1,-u_2}, & \quad 
& c_{u_1,-v_2}c_{s_1,-t_2} & = & c_{u_1,-t_2}c_{s_1,-v_2}, \\
c_{u_1,-v_2}c_{t_1,-s_2} & = & c_{u_1,-s_2}c_{t_1,-v_2} , & \quad 
& c_{u_1,-s_2}c_{v_1,-t_2} & = & c_{u_1,-t_2}c_{v_1,-s_2}, \\
 c_{v_1, -u_2}c_{t_1, -s_2} & = & c_{v_1, -s_2}c_{t_1, -u_2}, & \quad 
& c_{s_1, -u_2}c_{t_1, -v_2} & = & c_{s_1, -v_2}c_{t_1, -u_2}.
\end{array}
\right.
\end{equation*}
\end{lemma}

\begin{proof}
We establish below the first equality $c_{u_1,u_2} c_{v_1,v_2}= c_{u_1,-v_2}c_{v_1,-u_2}$. The same arguments can be used to establish the others.  
To proceed, we first note that the four roots $u_1$, $u_2$, $-v_1$, and $-v_2$ satisfy the assumption of Lemma~\ref{lem:auxiliary}. Thus,
$$
\frac{c_{u_1,u_2}c_{-v_1,-v_2}}{|u_1 + u_2|^2} + 
\frac{c_{u_2,-v_1}c_{u_1,-v_2}}{|u_2 - v_1|^2} + \frac{c_{-v_1,u_1}c_{u_2,-v_2}}{|u_1 - v_1|^2} = 0.
$$
Since $u_1-v_1$ and $u_2-v_2$ are not roots, $c_{-v_1,u_1} = c_{u_2,-v_2} = 0$. Next, note that $|u_1 + u_2|^2 = |u_1 - v_2|^2$. It then follows that $c_{u_1,u_2} c_{v_1,v_2}= c_{u_1,-v_2}c_{v_1,-u_2}$. 
\end{proof}

\begin{remark}\label{rmk:equalitiesforc}
There are a few other equalities of the $c_{*,*}$'s in~\eqref{eq:c**}: 
$$
\left\{
\begin{array}{lllllll}
c_{v_1, v_2}c_{s_1, s_2} & = & c_{v_1, -s_2}c_{s_1,-v_2}, & \quad 
& c_{v_1, v_2}c_{t_1, t_2} & = & c_{v_1, -t_2}c_{t_1,-v_2}, \\
c_{v_1, -u_2}c_{s_1, -t_2} & = & c_{s_1, -u_2}c_{v_1, -t_2}, & \quad 
& c_{s_1,s_2}c_{t_1, t_2} & = & c_{s_1, -t_2}c_{t_1, -s_2}, 
\end{array}
\right.
$$
which follow as consequences of Lemma~\ref{lem:equalitiesforc}. 
\end{remark}

With the preliminaries above, we will now compute $[\Sigma_1, \Psi_{u_2, v_2}]$.  
The values of $\epsilon_1$ and $\epsilon_2$ are specified as follows: For $\epsilon_1 = (\epsilon_{v_1}, \epsilon_{s_1}, \epsilon_{t_1})$, each of its entries can be either $1$ or $-1$; For $\epsilon_2 = (\epsilon_{v_2}, \epsilon_{s_2}, \epsilon_{t_2})$, the first entry $\epsilon_{v_2}$ can be either $1$ or $-1$ and the other two entires $\epsilon_{s_2}$ and $\epsilon_{t_2}$ are $0$. 
Since $\epsilon_{s_2} = \epsilon_{t_2} = 0$, 
the last two addends in the expressions of $\mu_i$, $\nu_i$ for $i = 1,2$ in~\eqref{eq:c**} and the first two addends in $\mu_i$, $\nu_i$ for $i = 3,4$ are all zero. 
Now, we define 
\begin{equation}\label{eq:defpq1}
\begin{array}{lll}
p_1:= \sigma_2 \epsilon_{v_1} c_{v_1, -u_2}, &  p_2 := \sigma_2 \epsilon_{s_1} c_{s_1, -u_2}, & p_3:=  \sigma_2\epsilon_{t_1} c_{t_1,-u_2}, \\
p_4:= \sigma_2 \epsilon_{v_2} c_{u_1, -v_2}, & q_1 :=c_{u_1, u_2}, & q_2 := c_{t_1,-u_2}c_{u_1,-v_2}c_{t_1,-v_2}. 
\end{array}
\end{equation}
Note that each $p_i$, for $i = 1,\ldots, 4$, is a free variable taking value from the set $\{1, -1\}$ while $q_1$ and $q_2$ are fixed. By~\eqref{eq:c**} and Lemma~\ref{lem:equalitiesforc}, we obtain the following expressions for the coefficients $\mu_i$ and $\nu_i$  in~\eqref{eq:sigmasigma}:  
\begin{equation}\label{eq:c**reduced}
 \left\{
\begin{array}{lll}
	\mu_1 = \nu_1 =  q_1 + q_1p_1p_4, & & \mu_2 = \nu_2 = p_1 + p_4, \\
	\mu_3  =  p_2 - q_2 p_3p_4, & & \nu_3  = p_2 + q_2 p_3p_4\\ 
	\mu_4  = p_3 - q_2p_2p_4, & & \nu_4 =  p_3 + q_2p_2p_4. 
\end{array}
\right.
\end{equation}
We can thus obtain from~\eqref{eq:c**reduced} all possible values of $\mu_i$ and $\nu_i$ by choosing different values of $p_i\in \{1,-1\}$ for $i = 1,\ldots, 4$.  
Examination of these values, combined with Lemma~\ref{lem:computesigmasigma}, leads to the fact that 
%
\begin{equation}\label{eq:sigmapsiresult}
[\Sigma_1, \Psi_{u_2, v_2}] \equiv \Sigma_{3} \cup \{Y_{u_3}\pm Y_{v_3}, Z_{u_3} \pm Z_{v_3}\} \cup \{Y_{s_3}\pm Y_{t_3}, Z_{s_3} \pm Z_{t_3}\}. 
\end{equation}
Finally, by item { (c)} of Theorem~\ref{prop:orthogonalpairs}, $(u_3,v_3)$ and $(s_3,t_3)$ are {\em iop}s. We thus conclude from~\eqref{eq:sigmapsiresult} that $[\Sigma_1, \Psi_{u_2, v_2}]\equiv \Sigma_3 \cup \Psi_{u_3,v_3} \cup \Psi_{s_3,t_3}$. \hfill{\qed}

\section{Proof of Prop.~\ref{prop:computesigmasigma}}
We establish below the two items.
\vspace{.1cm}

{\em Proof of item~1.}
Let $u_i$, $v_i$, $s_i$, and $t_i$ be defined in~\eqref{eq:defuvst}. Recall that $\Delta_{{\rm sin}_i} = \pm \{u_i, v_i, s_i, t_i\}$ for all $i = 1,2,3$. 
Because roots in $\Delta_{{\rm sin}_i}$ are either proportional or strongly orthogonal, we obtain that 
$[\Sigma_i,\Sigma_i] \equiv \{\mathrm{i} H_{u_i} \pm \mathrm{i} H_{v_i} \pm \mathrm{i} H_{s_i} \pm \mathrm{i} H_{t_i} \}$. 
The item then follows from the fact that $\{\pm u_i \pm v_i \pm s_i \pm t_i \}= 2(\Delta_{\rm sin} \backslash \Delta_{{\rm sin}_i})$. 

\vspace{.1cm}

{\em Proof of item~2.} 
By symmetry, we prove only for $[\Sigma_1, \Sigma_2]$. 
The proof is similar to the one for item~2 of Prop.~\ref{prop:computesigmapsi}.  
Let $\mathbf{Y}_i(\epsilon_i)$ and $\mathbf{Z}_i(\epsilon_i)$ be defined in~\eqref{eq:defsigmayz}  
and let each entry of $\epsilon_1$ or $\epsilon_2$ be either $1$ or $-1$.    
Lemmas~\ref{lem:defsigmai} -- \ref{lem:equalitiesforc} as well as Remark~\ref{rmk:equalitiesforc} still hold. In addition to the $p_i$, for $i = 1,\ldots, 4$, defined in~\eqref{eq:defpq1}, we further let 
$$
p_5 := \sigma_2\epsilon_{s_2} c_{u_1, -s_2} \quad \mbox{and} \quad   
p_6 := \sigma_2\epsilon_{t_2} c_{u_1, -t_2}.
$$
Then, all the $p_i$, for $i = 1,\ldots, 6$, can be treated as free variables which take value from the set $\{1,-1\}$ while $q_1$ and $q_2$ defined in~\eqref{eq:defpq1} are fixed. By~\eqref{eq:c**} and Lemma~\ref{lem:equalitiesforc}, we obtain the following expressions for the coefficients $\mu_i$ and $\nu_i$  in~\eqref{eq:sigmasigma}: 
\begin{equation}\label{eq:c**reduced1}
 \left\{
\begin{array}{lll}
	\mu_1 & = & \nu_1 \,\, = \,\,  q_1 + q_1p_1p_4 + q_1p_2p_5 + q_1 p_3p_6, \\
	\mu_2 & = & p_1 + p_4 - q_2 p_2p_6 - q_2 p_3 p_5, \\
	\mu_3 & = & p_2 + p_5  - q_2p_1 p_6 - q_2 p_3p_4, \\ 
	\mu_4 & = & p_3 + p_6 - q_2p_1p_5 - q_2p_2p_4, \\
	\nu_2 & = & p_1 + p_4 + q_2 p_2p_6 + q_2 p_3 p_5, \\
	\nu_3 & = & p_2 + p_5 + q_2p_1 p_6 + q_2 p_3p_4, \\  
	\nu_4 & = & p_3 + p_6 + q_2p_1p_5 + q_2p_2p_4. 
\end{array}
\right.
\end{equation}
We can thus obtain from~\eqref{eq:c**reduced1} all possible values of $\mu_i$ and $\nu_i$ by choosing different values of $p_i\in \{1,-1\}$ for $i = 1,\ldots, 6$.  
Examination of these values, combined with Lemma~\ref{lem:computesigmasigma}, leads to the fact that
$[\Sigma_1,\Sigma_2] \equiv \Sigma_3 \cup \bigcup_{x\in \Delta_{{\rm sin}_3}} \{Y_x, Z_x\}$.  
\hfill{\qed}

\section{Proof of Prop.~\ref{prop:computeg2}} 
We establish below the two items.
\vspace{.1cm}

{\em Proof of item~1.} For each $i = 1,2,3$, $x_i$ and $y_i$ are strongly orthogonal. It follows that $[\Phi_i, \Psi_i]\equiv [\Phi_i, \Phi_i]$ and, by computation, $[\Phi_i, \Phi_i]\equiv \Phi_i$. Next, for $[\Psi_i,\Psi_i]$,  we first obtain by computation that $[\Psi_i,\Psi_i]\equiv \{ \mathrm{i} (H_{y_i} \pm H_{x_i}) \}$. Let $(i,j,k)$ be a cyclic rotation of $(1,2,3)$. Then, $2y_j = -(y_i + 3x_i)$ and $2y_k = -(y_i - 3x_i)$. By the linearity $h_{\alpha + \beta} = h_\alpha + h_\beta$ and the scaling $H_\alpha = \nicefrac{2h_\alpha}{|\alpha|^2}$, we obtain that 
$$
\left\{
\begin{array}{lllll}
	2\mathrm{i} H_{y_j} & = & - \nicefrac{|y_i|^2}{|y_j|^2} \, \mathrm{i}H_{y_i} - \nicefrac{3|x_i|^2}{|y_j|^2}\, \mathrm{i} H_{x_i} & = & -\mathrm{i}(H_{y_i} + H_{x_i}), \vspace{3pt}\\ 
	2\mathrm{i} H_{y_k} & = & - \nicefrac{|y_i|^2}{|y_j|^2} \, \mathrm{i}H_{y_i} + \nicefrac{3|x_i|^2}{|y_j|^2}\, \mathrm{i} H_{x_i} & = & -\mathrm{i}(H_{y_i} - H_{x_i}).
\end{array}
\right.
$$
It then follows that $[\Psi_i, \Psi_i] \equiv \{\mathrm{i}H_{y_j}, \mathrm{i} H_{y_k}\}$. 
\vspace{.1cm}

{\em Proof of item~2.}
The results for $[\Phi_i, \Phi_j]$ and $[\Phi_i, \Psi_j]$ directly follow from computation. We establish the result for $[\Psi_i,\Psi_j]$. By symmetry, we can assume, without loss of generality, that $i = 1$ and $j = 2$. 
To proceed, we first have the following fact adapted from Lemma 5.1~\cite[Ch.~III]{SH:79}:

\begin{lemma}\label{lem:threerootsumzero}
If three roots $\alpha$, $\beta$, $\gamma$ sum to $0$, then $\nicefrac{c_{\alpha, \beta}}{|\gamma|^2} = \nicefrac{c_{\beta, \gamma}}{|\alpha|^2} = \nicefrac{c_{\gamma, \alpha}}{|\beta|^2}$.	
\end{lemma}

Note that $\sum^3_{i = 1}x_i = \sum^3_{i = 1}y_i = 0$ and, moreover,  $|x_1| = |x_2| = |x_3|$ and $|y_1| = |y_2| = |y_3|$. 
By Lemma~\ref{lem:threerootsumzero}, the following hold:
\begin{equation}\label{eq:n123}
c_{x_1, x_{2}} = c_{x_2,x_3} = c_{x_3,x_1} = \pm 2\quad
\mbox{and} 
\quad 
c_{y_1,y_2} = c_{y_2,y_3} = c_{y_3,y_1} = \pm 1.
\end{equation}
Note that $[X_{x_i}, X_{x_j}] = c_{x_i,x_j} X_{-x_k}$ and $[X_{y_i}, X_{y_j}] = c_{y_i,y_j} X_{-y_k}$. By negating (if necessary) $X_{x_1}$ and $X_{-x_1}$ (resp. $X_{y_1}$ and $X_{-y_1}$), we can assume that every $c_{x_i, x_j}$ (resp. $c_{y_i,y_j}$) in~\eqref{eq:n123} is positive. We next have the following fact:

\begin{lemma}
For any real numebrs $\epsilon$ and $\epsilon'$, the following hold: 	
\begin{equation}\label{eq:compforG2}
\left\{
\begin{array}{lll}
	{[Y_{y_1} + \epsilon Y_{x_1}, Y_{y_2} + \epsilon' Y_{x_2}]} & = & 
	\mu_1 Y_{y_3} + \mu_3 Y_{x_3}, \vspace{3pt}\\
	{[Z_{y_1} + \epsilon Z_{x_1}, Z_{y_2} + \epsilon' Z_{x_2}]} & = &
	\mu_2 Y_{y_3} + \mu_4 Y_{x_3}, \vspace{3pt}\\
	{[Y_{y_1} + \epsilon Y_{x_1}, Z_{y_2} + \epsilon' Z_{x_2}]} & = & 
	\nu_1 Z_{y_3} + \nu_3 Z_{x_3}, \vspace{3pt}\\
	{[Z_{y_1} + \epsilon Z_{x_1}, Y_{y_2} + \epsilon' Y_{x_2}]} & = & 
	\nu_2 Z_{y_3} + \nu_4 Z_{x_3},	
\end{array}
\right.
\end{equation}
where coefficients $\mu_i$ and $\nu_i$, for $i = 1,2,3,4$, are given by 
\begin{equation}\label{eq:munumunu}
\left\{
\begin{array}{lll}
\mu_1 & := & -\nu_1 \,\, := \,\, -(1 + \epsilon\epsilon' c_{x_1, -x_2}), \\
\mu_2 & := & \nu_2 \,\, := \,\, 1 - \epsilon\epsilon' c_{x_1, -x_2}, \\
\mu_3 & := & \epsilon c_{x_1,y_2} + \epsilon' c_{y_1,-x_2} - 2\epsilon\epsilon', \\
\nu_3 & := & \epsilon c_{x_1,y_2} + \epsilon' c_{y_1,-x_2} + 2\epsilon\epsilon', \\
\mu_4 & := & - \epsilon c_{x_1,y_2} + \epsilon' c_{y_1,-x_2} + 2\epsilon\epsilon', \\
\nu_4 & := & \epsilon c_{x_1,y_2} - \epsilon' c_{y_1,-x_2} + 2\epsilon\epsilon'. 
\end{array}
\right. 
\end{equation}  
\end{lemma}

\begin{proof}
Note that 
$
y_1 + y_2 = x_2 - x_1 = -y_3 
$ 
and
$
x_1 + y_2 = x_2 - y_1 = x_3
$. Also, note that $\g_0$ is the compact real form of $\g$ and, hence, $\sigma_\alpha = 1$ for all $\alpha\in \Delta$. 
The lemma then follows from~\eqref{lem3:cond1} and~\eqref{eq:n123}. 
\end{proof}

Note that $c^2_{x_1,-x_2} = 9$ and $c^2_{x_1,y_2} = c^2_{y_1,-x_2} = 1$. Furthermore, the following holds:

\begin{lemma}\label{lem:3c}
	We have that $c_{x_1,-x_2} = 3c_{x_1,y_2}c_{y_1,-x_2}$.
\end{lemma}

\begin{proof}
The four roots $x_1$, $-x_2$, $y_1$, $y_2$ satisfy the assumption of Lemma~\ref{lem:auxiliary}, i.e., $x_1 -x_2 + y_1 + y_2 = 0$ and no two of which sum to zero. Thus, 
$$
\frac{c_{x_1, -x_2}c_{y_1,y_2}}{|x_1-x_2|^2} + \frac{c_{-x_2,y_1}c_{x_1, y_2}}{|y_1-x_2|^2} + \frac{c_{y_1,x_1}c_{-x_2, y_2}}{|y_1 + x_1|^2} = 0.
$$
Since $y_1 + x_1$ and $y_2-x_2$ are not roots, $c_{y_1,x_1} = c_{-x_2, y_2} = 0$. Also, note  that $x_1 - x_2 = y_3$ and $y_1 - x_2 = -x_3$, so $|x_1 - x_2|^2 = 3|y_1 - x_2|^2$. 
Since $c_{y_1,y_2} = 1$, it follows that  
$c_{x_1,-x_2} = 3c_{x_1,y_2}c_{y_1,-x_2}$. 
\end{proof} 

In the sequel, we will let $\epsilon$ and $\epsilon'$ take values from the set $\{1,-1\}$. 
For convenience, we let $$p:= \epsilon c_{x_1, y_2}, \quad  p':= \epsilon' c_{y_1, -x_2}, \quad q:= \nicefrac{c_{x_1,-x_2}}{3}.$$ We treat $p$ and $p'$ as free variables which take value from the set $\{1,-1\}$. Then, by~\eqref{eq:munumunu} and Lemma~\ref{lem:3c}, the following hold for $\mu_i$ and $\nu_i$:
\begin{equation}\label{eq:munumunure}
\left\{
\begin{array}{lll}
\mu_1 & = & -\nu_1 \,\, = \,\, -(1 + 3pp'), \\
\mu_2 & = & \nu_2 \,\, = \,\, 1 - 3pp', \\
\mu_3 & = & p + p' - 2q pp', \\
\nu_3 & = & p + p' + 2q pp', \\
\mu_4 & = & - p + p' + 2q p p', \\
\nu_4 & = & p - p' + 2q pp'. 
\end{array}
\right. 
\end{equation} 
We obtain from the above expression all possible values of $\mu_i$ and $\nu_i$ by choosing different values of $p, p'\in \{1,-1\}$. Examination of these values, combined with~\eqref{eq:compforG2}, leads to the fact that $[\Psi_1, \Psi_2] \equiv \Psi_3 \cup \{Y_{y_3}, Z_{y_3}\}$.  \hfill{\qed}

\end{document}